\documentclass[leqno,12pt]{article}
\usepackage[usenames, dvipsnames]{xcolor}
\usepackage{tikz}

\usepackage{amssymb,hyperref,amsthm}
\usepackage{euscript}
\usepackage{flafter}
\usepackage{pstricks}
\usepackage[latin1]{inputenc}
\usepackage{amsmath}
\usepackage{amsfonts}
\usepackage{pstricks-add}
\usepackage{dsfont}
\usepackage{bm}
\usepackage{variations}

\hypersetup{
    linktoc=page,
   linkcolor=red,          
  citecolor=blue,        
    filecolor=blue,      
   urlcolor=cyan,
    colorlinks=true           
}

\usepackage[refpage]{nomencl}

\usepackage{makeidx}
\makeindex

\usepackage{mathrsfs} 

\evensidemargin\oddsidemargin

\newcommand{\eqnsection}{
\renewcommand{\theequation}{\thesection.\arabic{equation}}
   \makeatletter
   \csname  @addtoreset\endcsname{equation}{section}
   \makeatother}
\eqnsection

\def\r{{\mathbb R}}

\def\P{{\bf P}}
\def\E{{\bf E}}

\def\1{{\mathds{1}}}

\begin{document}

\baselineskip=18pt
\setcounter{page}{1}

\newtheorem{theorem}{Theorem}[section]

\newtheorem{Definition}[theorem]{Definition}
\newtheorem{lemma}[theorem]{Lemma}
\newtheorem{fact}[theorem]{Fact}
\newtheorem{proposition}[theorem]{Proposition}
\newtheorem{remark}[theorem]{Remark}
\newtheorem{conjecture}[theorem]{Conjecture}
\newtheorem{corollary}[theorem]{Corollary}
\newtheorem*{assa}{Assumption (A)}

\vglue50pt

\centerline{\large\bf Increasing paths on $N$-ary trees}

\bigskip
\bigskip
\centerline{Xinxin Chen}

\medskip

\centerline{\it LPMA, Universit\'e Paris VI}

\bigskip
\bigskip
\bigskip

\bigskip
{\leftskip=2truecm \rightskip=2truecm \baselineskip=15pt \small

\noindent{\slshape\bfseries Summary.} Consider a rooted $N$-ary tree. To every vertex of this tree, we attach an i.i.d. continuous random variable. A vertex is called accessible if along its ancestral line, the attached random variables are increasing. We keep accessible vertices and kill all the others. For any positive constant $\alpha$, we describe the asymptotic behaviors of the population at the $\alpha N$-th generation as $N$ goes to infinity. We also study the criticality of the survival probability at the $(eN-\frac{3}{2}\log N)$-th generation in this paper.

\bigskip

\bigskip

} 

\bigskip
\bigskip

\section{Introduction}
   \label{ss:InPintro}
   
\subsection{The model}
We consider an $N$-ary tree $T^{(N)}$, which is rooted at $\varnothing$, so that each vertex in $T^{(N)}$ has exactly $N$ children. To every vertex $\sigma\in T^{(N)}$, we assign a continuous random variable, denoted by $x_{\sigma}$. All these variables $x_\sigma$, $\sigma\in T^{(N)}$ are i.i.d. Let $|\sigma|$ denote the generation of $\sigma$, and $\sigma_i$ (for $0\le i\le |\sigma|$) denote its ancestor at generation $i$. The ancestral line of $\sigma$ is denoted by
\begin{equation*}
[\![ \varnothing, \, \sigma]\!]:=\{\sigma_0:=\varnothing,\sigma_1,\cdots, \sigma_{|\sigma|}:=\sigma\},
\end{equation*}
which is also the unique shortest path relating $\sigma$ to the root $\varnothing$. A vertex $\sigma$ is called accessible if along its ancestral line, the assigned random variables are increasing, i.e.,
\begin{equation}\label{InPaccessibility}
\sigma \text{ accessible }\Leftrightarrow x_{\varnothing}<x_{\sigma_1}<\cdots<x_\sigma.
\end{equation}
This model is called accessibility percolation by Nowak and Krug \cite{Nowak-Krug2013}. We also call $[\![\varnothing,\,\sigma]\!]$ an accessible path if $\sigma$ is accessible.

The model comes from evolutionary biology, in which both mutation and selection involve. As the main source of evolutionary novelty, mutations act on the genetic constitution of an organism. In our setting, each vertex represents one gene type, or genotype. A certain genotype may reproduce several new genotypes through mutations. The mechanism of successive mutations hence gives the structure of trees if we also assume that each mutation gives rise to a new genotype. Selection involves so that organisms better adapted to their respective surroundings are favored to survive. We suppose that each genotype (vertex) has an associated fitness value, which is represented by the assigned random variable. In the strong-selection/weak mutation regime, we assume that only mutations which give rise to a larger fitness value survive. In this way, the survival mutational pathways are noted by the accessible vertices. In this paper, we use `House of Cards' model (see \cite{Kingman1978}), in which all fitness values are i.i.d. As is explained in \cite{Franke-K-V-K2011}, it serves as a null model.

A variation of our model by replacing $N$-ary trees with $N$-dimensional hypercube has been considered in \cite{Berestycki-Brunet-Shi2013} and \cite{Hegarty-Martinsson2013}. More models are introduced in \cite{Aita-U-I-N-K-H2000} and \cite{Franke-K-V-K2011} to explain evolution via mutation and selection.

\subsection{Main results}

For any $k\geq1$, let $\mathcal{A}_{N,k}:=\{\sigma\in T^{(N)}: |\sigma|=k, \sigma\textrm{ is accessible}\}$. We define
\begin{equation}
Z_{N,k}:=\sum_{|\sigma|=k}1_{(\sigma\in\mathcal{A}_{N,k})}=\#\mathcal{A}_{N,k},\quad \forall k\geq1.
\end{equation}

Since we are only concerned with the order of the random variables, under the assumption of continuity of their law, changing the precise distribution will not influence the results. Without loss of generality, we assume throughout the paper that the assigned random variables are distributed uniformly in $[0,1]$, i.e., $\forall\sigma\in T^{(N)}$,  $x_{\sigma}$ has the uniform distribution in $[0,1]$, which is denoted by $U[0,1]$.

For any $x\in[0,1]$, we introduce the following probability measure:
\begin{equation}
\P_x(\cdot):=\P(\cdot\vert x_{\varnothing}=x).
\end{equation}

A nature question is about the survival probability $\P_x(Z_{N,k}\geq1)$. 
Note that for any $N,\, k\geq1$, $Z_{N,k}=\sum_{|\sigma|=k}1_{(\sigma\in\mathcal{A}_{N,k})}=\sum_{|\sigma|=k}1_{(x_{\varnothing}<x_{\sigma_1}<\cdots<x_\sigma)}$. We observe that
\begin{equation}
\E\Big[Z_{N,k}\Big]= N^k \P(x_{\varnothing}<x_{\sigma_1}<\cdots<x_\sigma)=\frac{N^k}{(k+1)!},
\end{equation}
since $x_{\varnothing}, x_{\sigma_1}, x_{\sigma_2}\cdots, x_\sigma$ are i.i.d. and distributed uniformly in $[0,1]$. Immediately,
\begin{equation}\label{eqch3:firstmom}
\E_x\Big[Z_{N,k}\Big]=N^k \P_x(x<x_{\sigma_1}<\cdots<x_\sigma\leq1)=N^k\frac{(1-x)^k}{k!},\quad \forall x\in[0,1].
\end{equation}
Stirling's approximation says that
\begin{equation}\label{InPeqstirling}
2<\frac{k!}{\sqrt{k} (k/e)^k}<3,\quad \forall k\geq1.
\end{equation}
It follows that 
\begin{equation}\label{eq:1mom}
\frac{1}{3\sqrt{k}}\Big(\frac{eN(1-x)}{k}\Big)^k\leq\E_x\Big[Z_{N,k}\Big]\leq\frac{1}{2\sqrt{k}}\Big(\frac{eN(1-x)}{k}\Big)^k.
\end{equation}
It is thus reasonable to take $k=\lfloor \alpha N\rfloor$ with $\alpha>0$. 

For convenience, we write $\alpha N$ to represent the integer $\lfloor \alpha N\rfloor$ throughout this paper. We are interested in the asymptotic behaviors of $Z_{N,\alpha N}$ as $N\rightarrow\infty$.

Nowak and Krug \cite{Nowak-Krug2013} showed that $\liminf_{N\rightarrow\infty}\P_0[Z_{N,\alpha N}\geq1]>0$ for $0<\alpha<1$ and that $\lim_{N\rightarrow\infty}\P_0[Z_{N,\alpha N}\geq1]=0$ for $\alpha\geq e$. This transition of phases implies the existence of a critical value of $\alpha$. Roberts and Zhao \cite{Roberts-Zhao2013} proved that the critical value is $\alpha_c=e$, by considering some typical increasing paths. In fact, we have
\begin{equation}
\lim_{N\rightarrow\infty}\P_0\Big(Z_{N,\alpha N}\geq1\Big)=\left\{\begin{array}{cl}
1 & \textrm{if }\alpha < e;\\
0 & \textrm{if }\alpha \geq e.
\end{array}\right.
\end{equation}

This result tells us that, for $N$ large, roughly speaking, the population of accessible vertices survives until the $e N$-th generation and then dies out. Let us describe the asymptotic behaviors of the population more precisely by the following theorems.

\begin{theorem}\label{InPlimitcurve}
Let $\theta(\alpha):=\alpha(1-\log \alpha)$ for $\alpha>0$.
\begin{itemize}
  \item [(i)] When $\alpha\in(0, e)$, the following convergence holds $\P_0-$almost surely,
 \begin{equation}
  \lim_{N\rightarrow\infty}\frac{Z_{N,\alpha N}}{N}=\theta(\alpha)>0.
  \end{equation}

  \item [(ii)] When $\alpha=e$, we have
  \begin{equation}
  \P_0\Big(Z_{N,\alpha N}\geq1\Big)=N^{-3/2+o_N(1)}\textrm{ as }N\rightarrow\infty,
  \end{equation}
  where $\{o_N(1)\}_{N\geq1}$ is a sequence of real numbers which goes to zero as $N\rightarrow\infty$.

  \item [(iii)] When $\alpha>e$, we have
  \begin{equation}
  \lim_{N\rightarrow\infty}\frac{\log \P_0\Big(Z_{N,\alpha N}\geq1\Big)}{N}=\theta(\alpha)<0.
  \end{equation}

\end{itemize}
\end{theorem}

\begin{remark}
For $\alpha<e$, the accessible population $Z_{N,\alpha N}$ is exponentially large. Its second order is not given here, but we present some arguments in Appendix \ref{ss:secondorder}. When $\alpha=e$, the explicit order of the survival probability is still unknown. 
\end{remark}
It is clear that the system becomes extinct before the generation $eN$. In the next theorem, we see that the real critical generation is $eN-\frac{3}{2}\log N$.

\begin{theorem}\label{thm:critical}
Let $k=e N- \beta \log N$. Then we have
\begin{equation}
\lim_{N\rightarrow\infty}\P_0\Big(Z_{N,k}\geq1\Big)=\left\{\begin{array}{cl}
1 & \textrm{if }\beta >3/2;\\
0 & \textrm{if }\beta <3/2.
\end{array}\right.
\end{equation}
\end{theorem}

At the critical generation $k=eN-\frac{3}{2}\log N$, the survival probability is not clear at this moment. We state the following proposition, which only gives a lower bound.
\begin{proposition}\label{prop:critical}
For any $\varepsilon>0$ and $n$ sufficiently large, we have
\begin{equation}
\P_0\Big(Z_{N,eN-\frac{3}{2}\log N}\geq1\Big)\geq N^{-\varepsilon}.
\end{equation}
\end{proposition}

It is possible to replace the $N$-ary tree by the Galton-Watson tree whose offspring is Poisson with parameter $N$, in which case all these results still hold.

The rest of the paper is organized as follows. In Section \ref{ss:InPphase-transition}, we state some basic results of the accessible population and the increasing paths. In Section \ref{ss:InPlimitcurve} we prove Theorem \ref{InPlimitcurve}. Finally, in Section \ref{ss:critical}, we show the criticality at $eN-\frac{3}{2}\log N$, by proving Theorem \ref{thm:critical} and \ref{prop:critical}.

Throughout the paper, we use the letter $c$ with subscript to denote a finite and positive constant.

\section{Basic ideas of the increasing paths}
 \label{ss:InPphase-transition}

\subsection{The generating function of $Z_{N,k}$}
As $Z_{N,k}$ is an integer-valued r.w., we consider its generation function $\E_x\Big(s^{Z_{N,k}}\Big)$ in this subsection.

Generally, for any $0\leq a<b\leq1$, we define $Z_{N,k}(a,b)$ as follows:
\begin{equation}\label{InPpartialpopulation}
Z_{N,k}(a,b):=\sum_{|\sigma|=k} 1_{(a<x_{\sigma_1}<\cdots<x_{\sigma_k}\leq b)},\quad\forall k\geq1.
\end{equation}
For convenience, we write $Z_{N,k}(b)$ for $Z_{N,k}(0,b)$ and set $Z_{N,0}(b)\equiv1$. One sees that $Z_{N,k}(b-a)$ and $Z_{N,k}(a,b)$ have the same law.  Let $f_k^{(N)}(s,b)$ be the generating function of $Z_{N,k}(b)$, i.e.,
\begin{equation}\label{InPgeneratingfuncpartialpopulation}
f_k^{(N)}(s,b):=\E\Big[s^{Z_{N,k}(b)}\Big],\quad \forall s\in[0,1].
\end{equation}
For $k=1$, $Z_{N,1}(b)$ is a binomial variable with parameter $(N, b)$. So $f_1^{(N)}(s,b)=[1-b+sb]^N$.

For any $k\geq 1$, one observes that
\begin{equation}
Z_{N,k+1}(b)=\sum_{|\sigma|=1}1_{(x_\sigma<b)}\sum_{|\omega|=k+1}1_{(\omega_1=\sigma)}1_{(x_\sigma<x_{\omega_2}<\cdots<x_\omega\leq b)}.
\end{equation}
For all vertices $\sigma$ of the first generation, the variables $1_{(x_\sigma<b)}\sum_{|\omega|=k+1}1_{(\omega_1=\sigma)}1_{(x_\sigma<x_{\omega_2}<\cdots<x_\omega\leq b)}$ are i.i.d., and given $\{x_\sigma=y<b\}$, $\sum_{|\omega|=k+1}1_{(\omega_1=\sigma)}1_{(x_\sigma<x_{\omega_2}<\cdots<x_\omega\leq b)}$ is distributed as $Z_{N,k}(y,b)$. It follows that for any $k\geq1$ and any $b\in[0,1]$,
\begin{equation}\label{InPrecursioneq}
\begin{split}
f_{k+1}^{(N)}(s,b)=&\Big[1-b+\int_0^b dy\E\Big(s^{Z_{N,k}(y,b)}\Big)\Big]^N\\
=&\Big[1-b+\int_0^b f_k^{(N)}(s,b-y)dy\Big]^N\\
=&\Big[1-b+\int_0^b f_k^{(N)}(s,y)dy\Big]^N.
\end{split}
\end{equation}
For brevity, we denote the generating function of $Z_{N,k}$ under $\P_0$ by $f_k^{(N)}(s)$ instead of $f_k^{(N)}(s,1)$. Immediately,
\begin{equation}
f_{k+1}^{(N)}(s)=\Big[\int_0^1 f_k^{(N)}(s,y)dy\Big]^N,\quad\forall k\geq1.
\end{equation}
This gives that
\begin{equation}
\P_0\Big(Z_{N,k}\geq1\Big)=1-f_k^{(N)}(0)=1-\Big[\int_0^1 f_{k-1}^{(N)}(0,y)dy\Big]^N,\quad\forall k\geq2.
\end{equation}
To study the law of $Z_{N,k}$, it suffices to study $(\ref{InPrecursioneq})$. However, it is quite difficult to investigate analytically the sequence $f_k^{(N)}$, $k\geq1$, from the recursive relation (\ref{InPrecursioneq}). We thus turn to study the accessible vertices via their paths. 

\subsection{Typical accessible paths}
To study an increasing path, we let $\{U_j; j\geq1\}$ be a sequence of i.i.d. $U[0,1]$ random variables. Observe that for any $1\leq j\leq k$,
\begin{equation}
\E\Big(U_j\Big\vert U_1\leq U_2\leq \cdots \leq U_k\Big) = \frac{j}{k+1}.
\end{equation}
This leads us to comparing an increasing path $\{U_1\leq U_2\leq \cdots \leq U_k\}$ with the line $\{ \frac{j}{k+1}; 1\leq j\leq k\}$. For example, in Lemma 2 of \cite{Roberts-Zhao2013}, the authors showed that 
$$\P\Big(U_1\leq U_2\leq \cdots \leq U_k; U_j\geq\frac{j}{k+1},\forall 1\leq j\leq k\Big)=\frac{1}{(k+1)!}.$$

In what follows, we generalize their ideas and state the two lemmas, which estimate the probabilities of some typical accessible paths.
\begin{lemma}\label{InPtypicalpaths}
\begin{itemize}
  \item[(1)] For any $1\leq k\leq J-1$,
\begin{equation}\label{InPtypicalpath1}
\phi(k,J):=\P\left(U_1\leq \cdots\leq U_k; U_j\geq\frac{j}{J},\forall 1\leq j\leq k\right)=\frac{J-k}{k!J}.
\end{equation}
\item[(2)] For any $\varepsilon\in[0,1)$ and $1\leq k\leq J$,
\begin{equation}\label{InPtypicalpath2}
\begin{split}
\psi(k,J,\varepsilon):=&\P\left(U_1\leq \cdots\leq U_k; U_j\geq \varepsilon+(1-\varepsilon)\frac{j-1}{J},\forall 1\leq j\leq k\right)\\
=&\frac{(1+1/J)^k(J+1-k)}{k!(J+1)}(1-\varepsilon)^k.
\end{split}
\end{equation}
\end{itemize}
\end{lemma}

\begin{proof}[Proof.]
According to the assumption, we compute $\phi(k,J)$ directly.
\begin{equation*}
\begin{split}
\phi(k,J)=& \int_{k/J}^1\int_{(k-1)/J}^{u_k}\cdots\int_{1/J}^{u_2}du_1\cdots du_k\\
=&\int_{k/J}^1\int_{(k-1)/J}^{u_k}\cdots\int_{j/J}^{u_{j+1}}\Big(\frac{u^{j-1}_j}{(j-1)!}-\frac{1}{J}\frac{u_j^{j-2}}{(j-2)!}\Big)du_j\cdots du_k\\
=&\frac{J-k}{k!J},
\end{split}
\end{equation*}
giving (\ref{InPtypicalpath1}). 

We now compute $\psi$ by using $\phi$. Rewrite $\psi(k,J,\varepsilon)$ as follows:
\begin{eqnarray*}
\psi(k,J,\varepsilon)&=&\int_{\varepsilon+(1-\varepsilon)\frac{k-1}{J}}^1\cdots\int_{\varepsilon}^{u_2}du_1\cdots du_k.
\end{eqnarray*}
Take $u_j=\varepsilon+(1-\varepsilon)v_j$ for all $1\leq j\leq k$. By a change of variables,
\begin{equation}\label{eq:condInP}
\psi(k,J,\varepsilon)=\int_{(k-1)/J}^1\cdots\int_0^{v_2}(1-\varepsilon)^kdv_1\cdots dv_k=(1-\varepsilon)^k\psi(k,J,0).
\end{equation}
In particular, when $\varepsilon=\frac{1}{J+1}$, we have
$$
\psi(k,J,\frac{1}{J+1})=\psi(k,J,0)(1-\frac{1}{J+1})^{k}.
$$
On the other hand, when $\varepsilon=\frac{1}{J+1}$, $\varepsilon+(1-\varepsilon)\frac{j-1}{J}=\frac{j}{J+1}$ for any $1\leq j\leq k$. Hence,
\begin{equation*}\label{InPrelatingeq}
\psi(k,J,\frac{1}{J+1})=\phi(k,J+1)=\frac{J+1-k}{k!(J+1)}.
\end{equation*}
It follows that $\psi(k,J,0)=\psi(k,J,\frac{1}{J+1})(1-\frac{1}{J+1})^{-k}=\frac{(1+1/J)^k(J+1-k)}{k!(J+1)}$. By (\ref{eq:condInP}), we then obtain that
\begin{equation}
\psi(k,J,\varepsilon)=(1-\varepsilon)^k\psi(k,J,0)=\frac{(1+1/J)^k(J+1-k)}{k!(J+1)}(1-\varepsilon)^k,
\end{equation}
as desired.
\end{proof}

Following the assumption of Lemma \ref{InPtypicalpaths}, we define for any $0\leq L<K$,
\begin{equation}\label{InPdefofA}
A_L(K):=\left\{U_1< \cdots< U_K; U_{j}\geq \frac{(j-L)_+}{K+1}; \forall 1\leq j\leq K\right\}.
\end{equation}
Obviously, $\P[A_0(K)]=\phi(K,K+1)=\frac{1}{(K+1)!}$ by (\ref{InPtypicalpath1}).

\begin{lemma}\label{InPineqofA}
There exists a positive constant $c_0>0$ such that for any $1\leq L<K$,
\begin{equation}\label{InPuppofA}
\P\Big[A_L(K)\Big]\leq\frac{e^{c_0\sqrt{L}}}{K^{3/2}}\frac{e^K}{(K+1)^K}.
\end{equation}
\end{lemma}

\begin{proof}[Proof.]
Clearly, $\P[A_1(K)]=\psi(K,K,0)=\frac{(1+1/K)^K}{(K+1)!}$ by (\ref{InPtypicalpath2}). By (\ref{InPeqstirling}),
\begin{equation}
\P\Big[A_L(K)\Big]\leq\frac{e^{2\sqrt{L}}}{K^{3/2}}\frac{e^K}{(K+1)^K},\textrm{ for } L=1.
\end{equation}
The fact $A_1(K)\subset A_2(K)\subset\cdots\subset A_L(K)$ leads to
\begin{equation}\label{InPsumofevents}
\P[A_L(K)]=\sum_{i=1}^{L-1}\P[A_{i+1}(K)\setminus A_i(K)]+\P[A_1(K)],\quad 2\leq L<K.
\end{equation}
Let us estimate $\P[A_{i+1}(K)\setminus A_i(K)]$. Observe that
\begin{equation}\label{InPdifofevents}
\P[A_{i+1}(K)\setminus A_i(K)]=\sum_{k=i+1}^K\P[C_{i,k}(K)],
\end{equation}
where
\begin{equation}
C_{i,k}(K):=\bigg\{%
\begin{array}{ll}
U_1<\cdots< U_K; U_j&\geq \frac{j-i}{K+1},\forall i+1\leq j\leq k-1;\\
& U_k<\frac{k-i}{K+1}; U_{j}\geq\frac{j-i-1}{K+1},\forall k+1\leq j\leq K
 \end{array}
 \bigg\}.
\end{equation}
It follows from the independence of $U_j$'s that $\P[C_{i,k}(K)]=p_{i,k}q_{i,k}$ where
\begin{eqnarray*}
p_{i,k}&:=& \P\bigg(U_1<\cdots< U_k<\frac{k-i}{K+1}; U_j\geq \frac{j-i}{K+1},\forall i+1\leq j\leq k-1\bigg);\\
q_{i,k}&:=& \P\bigg(\frac{k-i}{K+1}\leq U_{k+1}< \cdots < U_K; U_{j}\geq\frac{j-i-1}{K+1},\forall k+1\leq j\leq K\bigg).
\end{eqnarray*}
Then (\ref{InPdifofevents}) becomes that
\begin{equation}\label{InPdifofevents1}
\P[A_{i+1}(K)\setminus A_i(K)]=\sum_{k=i+1}^Kp_{i,k}q_{i,k}.
\end{equation}
We first compute $q_{i,k}$:
\begin{eqnarray*}
q_{i,k}&=&\P\bigg(U_{1}< \cdots < U_{K-k}; U_{j}\geq\frac{j+k-i-1}{K+1},\forall 1\leq j\leq K-k\bigg)\\
&=&\psi(K-k, K-k+i+1, \frac{k-i}{K+1}).
\end{eqnarray*}
By (\ref{InPtypicalpath2}), we obtain that
\begin{equation}\label{InPeqofq}
q_{i,k}=\Big(\frac{K+2+i-k}{K+1}\Big)^{K-k} \frac{i+2}{(K-k)!(K-k+i+2)}.
\end{equation}
It remains to estimate $p_{i,k}$. One sees that
\begin{eqnarray}
p_{i,k}&=&\Big(\frac{k-i}{K+1}\Big)^k\P\bigg(U_1<\cdots< U_k; U_j\geq\frac{j-i}{k-i}, \forall i+1\leq j\leq k-1\bigg)\nonumber\\
&\leq&\Big(\frac{k-i}{K+1}\Big)^k\frac{1}{k-i}\P(D_{i,k-1}),
\end{eqnarray}
where
\begin{equation}
D_{i,k}:=\Big\{U_1<\cdots< U_k, U_j\geq \frac{j-i}{k-i+1},\forall i+1\leq j\leq k\Big\},\  k\geq i\geq1.
\end{equation}

Let us admit for the moment the following lemma, whose proof will be given later.
\begin{lemma}\label{InPineqofu}
For $k\geq i\geq 1$, there exists a constant $c_1>0$ such that
\begin{equation}\label{InPkeyineg}
u_{i,k}:=\P\Big(D_{i,k}\Big)\leq \frac{e^{k-i}e^{c_1\sqrt{i-1}+2}}{(k+1-i)^k k^{3/2}}.
\end{equation}
\end{lemma}

Lemma \ref{InPineqofu} implies that
\begin{eqnarray}\label{InPineqofp}
p_{i,k}&\leq& \Big(\frac{k-i}{K+1}\Big)^k\frac{1}{k-i}u_{i,k-1}\nonumber\\
&\leq&\Big(\frac{e}{K+1}\Big)^k\frac{e^{-i-1}e^{c_1\sqrt{i-1}+2}}{(k-1)^{3/2}}.
\end{eqnarray}
Let us go back to (\ref{InPdifofevents1}). In view of (\ref{InPeqofq}) and (\ref{InPineqofp}), we see that
\begin{multline}
\P[A_{i+1}(K)\setminus A_i(K)]=\sum_{k=i+1}^K p_{i,k}q_{i,k}\\
\leq\sum_{k=i+1}^K\Big(\frac{K+2+i-k}{K+1}\Big)^{K-k} \frac{i+2}{(K-k)!(K-k+i+2)}\Big(\frac{e}{K+1}\Big)^k\frac{e^{-i-1}e^{c_1\sqrt{i-1}+2}}{(k-1)^{3/2}}.
\end{multline}
Applying Stirling's formula (\ref{InPeqstirling}) to $(K-k)!$ yields that
\begin{eqnarray*}
\P[A_{i+1}(K)\setminus A_i(K)]&\leq&\frac{(i+2)e^{c_1\sqrt{i-1}+2}e^K}{(K+1)^K}\sum_{k=i}^{K-1}\frac{e}{2k^{3/2}(K+i+1-k)^{3/2}}\\
&\leq&c_2\frac{\sqrt{i}e^{c_1\sqrt{i-1}+2}}{(K+1)^{3/2}}\frac{e^K}{(K+1)^K}.
\end{eqnarray*}
We then deduce from (\ref{InPsumofevents}) that for $L\geq2$,
\begin{equation}
\P[A_L(K)]=\sum_{i=1}^{L-1}\P[A_{i+1}(K)\setminus A_i(K)]+\P[A_1(K)]\leq c_3\frac{L^{3/2}e^{c_1\sqrt{L-1}+2}}{K^{3/2}}\frac{e^K}{(K+1)^K},
\end{equation}
which is sufficient to conclude Lemma \ref{InPineqofA}.
\end{proof}

We now present the proof of Lemma \ref{InPineqofu}.

\begin{proof}[{Proof of Lemma \ref{InPineqofu}.}]
Recall that $D_{i,k}=\Big\{U_1<\cdots< U_k, U_j\geq \frac{j-i}{k-i+1},\forall i+1\leq j\leq k\Big\}$. Since $D_{i,k}\subset\{U_1<\cdots<U_k\}$, we have for any $k\geq i$,
\begin{equation}
u_{i,k}=\P\Big(D_{i,k}\Big)\leq \P\Big(U_1<\cdots<U_k\Big)=\frac{1}{k!}.
\end{equation}
By Stirling's formula (\ref{InPeqstirling}), we get that
\begin{equation*}
u_{i,k}\leq \frac{e^k}{2 k^k \sqrt{k}}=\frac{e^k}{(k+1-i)^k k^{3/2}}\Big(1-\frac{i-1}{k}\Big)^k\frac{k}{2}\leq \frac{e^{k+1-i}}{(k+1-i)^k k^{3/2}}\frac{k}{2},
\end{equation*}
as $1-z\leq e^{-z}$ for any $z\geq0$. Take $c_1:=\max\{40, \sup_{i\geq 2}\frac{1+\log i}{\sqrt{i-1}}\}<\infty$. Then when $k\leq 2i$, we deduce that
\begin{equation}
u_{i,k}\leq \frac{e^{k-i}}{(k+1-i)^k k^{3/2}}e^{\log i+1}\leq \frac{e^{k-i}e^{c_1\sqrt{i-1}+2}}{(k+1-i)^k k^{3/2}}.
\end{equation}

It remain to prove the inequality (\ref{InPkeyineg}) when $k/2\geq i\geq 1$. Let $\gamma(i):=e^{c_1\sqrt{i-1}+2}$. According to Lemma \ref{InPtypicalpaths}, we have
\begin{equation}
u_{1,k}=\psi(k,k,0)=\frac{(1+\frac{1}{k})^k}{(k+1)!}\leq \frac{e^{k+1}}{2 k^{k+3/2}}\leq \frac{e^{k-1}\gamma(1)}{(k+1-1)^k k^{3/2}},\ \forall k\geq1,
\end{equation}
giving (\ref{InPkeyineg}) in case $i=1$.

We prove (\ref{InPkeyineg}) by induction on $i$. Assume (\ref{InPkeyineg}) for some $i\geq1$ (and all $k\geq i$). We need to bound $\P(D_{i+1,k})$ for $k\geq 2(i+1)$.

Since $D_{1,k}\subset D_{2,k}\subset\cdots D_{k-1,k}$, we have:
\begin{multline}
u_{i+1,k}-u_{i,k}=\P\Big(D_{i+1,k}\setminus D_{i,k}\Big)\\
= \sum_{j=1}^{k-i}\P\bigg(U_1< \cdots< U_k, U_{i+\ell}\geq\frac{1}{k-i+1},\forall 1\leq\ell<j; \frac{j-1}{k-i}\leq U_{i+j}<\frac{j}{k-i+1};\\
\frac{j}{k-i+1}< \frac{j+\ell-1}{k-i}\leq U_{i+j+\ell},\forall 1\leq \ell\leq k-j-i\bigg).
\end{multline}
By the independence of the $U_i$'s, we have $u_{i+1,k}-u_{i,k}=\sum_{j=1}^{k-i}r_{i,j,k}s_{i,j,k}$ where
\begin{eqnarray*}
r_{i,j,k}:&=& \P\bigg(U_1< \cdots< U_{i+j}, U_{i+\ell}\geq\frac{\ell}{k-i+1},\forall 1\leq\ell<j; \frac{j-1}{k-i}\leq U_{i+j}<\frac{j}{k-i+1}\bigg)\\
s_{i,j,k}:&=& \P\bigg(U_{i+j+1}<\cdots< U_k; \frac{j+\ell-1}{k-i}\leq U_{i+j+\ell},\forall 1\leq \ell\leq k-j-i\bigg).
\end{eqnarray*}
Once again by (\ref{InPtypicalpath2}),
\begin{equation}
s_{i,j,k}=\psi(k-i-j,k-i-j,\frac{j}{k-i})=\bigg(\frac{k-i-j+1}{k-i}\bigg)^{k-i-j}\frac{1}{(k-i-j+1)!}.
\end{equation}
On the other hand,
\begin{eqnarray*}
r_{i,j,k}&\leq& \P\bigg(U_1\leq\cdots\leq U_{i+j-1}\leq \frac{j}{k-i+1},U_{i+\ell}\geq\frac{\ell}{k-i+1},\forall 1\leq\ell<j\bigg)\times\bigg[\frac{j}{k-i+1}-\frac{j-1}{k-i}\bigg]\\
&=& \bigg(\frac{j}{k-i+1}\bigg)^{i+j-1}\P\bigg(U_1\leq\cdots \leq U_{i+j-1}, U_{i+\ell}\geq \frac{\ell}{j},\forall 1\leq \ell\leq j-1\bigg)\frac{k-i-j+1}{(k-i)(k-i+1)}\\
&=&\bigg(\frac{j}{k-i+1}\bigg)^{i+j-1}\frac{k-i-j+1}{(k-i)(k-i+1)} u_{i,i+j-1}.
\end{eqnarray*}
This implies that
\begin{multline}
u_{i+1,k}-u_{i,k}=\sum_{j=1}^{k-i}r_{i,j,k}s_{i,j,k}\\
\leq\sum_{j=1}^{k-i}\bigg(\frac{k-i-j+1}{k-i}\bigg)^{k-i-j}\frac{1}{(k-i-j+1)!}\bigg(\frac{j}{k-i+1}\bigg)^{i+j-1}\frac{k-i-j+1}{(k-i)(k-i+1)} u_{i,i+j-1}.
\end{multline}

By induction assumption, for any $\ell\geq i\geq1$, $u_{i,\ell} \leq \frac{e^{\ell-i}\gamma(i)}{(\ell+1-i)^\ell \ell^{3/2}}$. It follows that
\begin{multline*}
u_{i+1,k}\leq \frac{e^{k-i}\gamma(i)}{(k+1-i)^k k^{3/2}}+\sum_{j=1}^{k-i}\bigg(\frac{j}{k-i+1}\bigg)^{i+j-1}\frac{k-i-j+1}{(k-i)(k-i+1)}\\
\times\frac{e^{j-1}\gamma(i)}{j^{i+j-1}(i+j-1)^{3/2}}\bigg(\frac{k-i-j+1}{k-i}\bigg)^{k-i-j}\frac{1}{(k-i-j+1)!}.
\end{multline*}
The first term on the right-hand side of this inequality is bounded by
\begin{equation}
\frac{e^{k-i}\gamma(i)}{(k-i)^k k^{3/2}}\Big(\frac{k-i}{k+1-i}\Big)^{k+1-i}\leq \frac{e^{k-i-1}\gamma(i)}{(k-i)^k k^{3/2}},
\end{equation}
whereas the second term bounded by
\begin{eqnarray*}
&&\sum_{j=1}^{k-i}\bigg(\frac{1}{k-i}\bigg)^{k+1}\frac{e^{j-1}\gamma(i)}{(i+j-1)^{3/2}}\bigg(k-i-j+1\bigg)^{k-i-j+1}\frac{1}{(k-i-j+1)!}\\
&\leq&\frac{e^{k-i}\gamma(i)}{(k-i)^{k+1}}\sum_{j=1}^{k-i}\frac{1}{2(i+j-1)^{3/2}(k-i-j+1)^{1/2}}\\
&\leq& \frac{20\gamma(i)}{\sqrt{i}}\frac{e^{k-i-1}}{(k-i)^{k}k^{3/2}},
\end{eqnarray*}
where the last inequality holds as we take $k/2\geq i+1$. We obtain that
\begin{equation}
u_{i+1,k}\leq \frac{e^{k-i-1}\gamma(i)}{(k-i)^k k^{3/2}}\Big(1+\frac{20}{\sqrt{i}}\Big)\leq \frac{e^{k-i-1}}{(k-i)^k k^{3/2}} \gamma(i)e^{\frac{20}{\sqrt{i}}}.
\end{equation}
Note that $\gamma(i)e^{\frac{20}{\sqrt{i}}}=\exp\{c_1\sqrt{i-1}+2+\frac{20}{\sqrt{i}}\}\leq \gamma(i+1)$ if we take $c_1>40$. Therefore,
\begin{equation}
u_{i+1,k}\leq \frac{e^{k-i-1}\gamma(i+1)}{(k-i)^k k^{3/2}},\  \forall k\geq i+1,
\end{equation}
which completes the proof of Lemma \ref{InPineqofu}.
\end{proof}

\section{Asymptotic behaviors of $Z_{N,\alpha N}$: Proof of Theorem \ref{InPlimitcurve}}
\label{ss:InPlimitcurve}

In this section, we prove Theorem \ref{InPlimitcurve}, by estimating the first and second moments of the accessible population. However, we do not consider directly $Z_{N,\alpha N}$ even though its second moment for $\alpha<2$ is obtained in Lemma \ref{InPsecondmoment}. In fact, we mainly count some typical increasing paths. 

For any $\varepsilon\in(0,1)$ and any $k\geq1$, let $\mathcal{A}_{N,k,\varepsilon}:=\{\sigma\in T^{(N)}: |\sigma|=k, x_{\sigma_1}<\cdots<x_\sigma; x_{\sigma_i}\geq \varepsilon+(1-\varepsilon)\frac{i-1}{k}, \forall 1\leq i\leq k\}$. We define the following quantities:
\begin{equation}\label{InPbar}
Z_{N,k,\varepsilon}:=\sum_{|\sigma|=k}1_{(\sigma\in\mathcal{A}_{N,k,\varepsilon})},\quad\forall k\geq1.
\end{equation}
Clearly, under $\P_0$, $Z_{N,k,\varepsilon}\leq Z_{N,k}=\#\mathcal{A}_{N,k}$. Instead of $Z_{N,k}$, we study $Z_{N,k,\varepsilon}$ with suitable $\varepsilon\geq0$.

\begin{proof}[\noindent\textit{Proof of (i) of Theorem \ref{InPlimitcurve}.}]
We need to show that for $\alpha\in(0,e)$, 
$$\P_0-a.s., \lim_{N\rightarrow\infty}\log Z_{N,\alpha N}/N=\theta(1-\alpha),$$
with $\theta(\alpha)=\alpha(1-\log \alpha)$. We first give the upper bound.
By (\ref{eq:1mom}),
\begin{equation}
\E_0\Big[Z_{N,\alpha N}\Big]\leq \frac{(e/\alpha)^{\alpha N}}{2\sqrt{\alpha N}}=\frac{e^{\theta(\alpha)N}}{2\sqrt{\alpha N}}.
\end{equation}
By Markov's inequality, for any $\delta>0$,
\begin{equation}
\P_0\Big[Z_{N,\alpha N}\geq \exp\{N(\theta(\alpha)+\delta)\}\Big]\leq \exp\{-N(\theta(\alpha)+\delta)\}\E_0\Big[Z_{N,\alpha N}\Big]\leq \frac{e^{-\delta N}}{2\sqrt{\alpha N}},
\end{equation}
which is summable in $N$. By the Borel-Cantelli lemma, for any $\delta>0$, $\P_0$-almost surely,
\begin{equation}
\limsup_{N\rightarrow\infty}\frac{\log Z_{N,\alpha N}}{N}\leq \theta(\alpha)+\delta.
\end{equation}
This establishes the upper bound. To obtain the lower bound, it suffices to show that for any $\delta>0$, there exists some $\varepsilon>0$ such that $\P_0$-almost surely,
\begin{equation}
\liminf_{N\rightarrow\infty}\frac{\log Z_{N,\alpha N}}{N}\geq\theta(\alpha)-\delta.
\end{equation}

By (\ref{InPtypicalpath2}), we see that for any $k\geq1$ and any $\varepsilon\in(0,1)$,
\begin{equation}\label{InPfirstmom}
\E_0\Big[Z_{N,k,\varepsilon}\Big]=N^{k}\psi(k,k,\varepsilon)= N^{k}\frac{(1+1/k)^k}{(k+1)!}(1-\varepsilon)^k.
\end{equation}
Here we take $k=\alpha N-1$ with $\alpha<e$. For any $\alpha<e$ fixed, take $\varepsilon$ small enough so that $\alpha <e(1-\varepsilon)$, $\theta(\alpha)>3\alpha\varepsilon$ and $\log (1-\varepsilon)>-2\varepsilon$. By Stirling's formula (\ref{InPeqstirling}),
\begin{equation}
\E_0\Big[Z_{N,\alpha N-1,\varepsilon}\Big]\geq c_4\frac{\exp\Big\{\theta(\alpha)N+\alpha\log (1-\varepsilon) N\Big\}}{(\alpha N)^{3/2}}.
\end{equation}
For all $N$ sufficiently large, we get that
\begin{equation}\label{eq:lowerbound1mom}
\E_0\Big[Z_{N,\alpha N-1,\varepsilon}\Big]\geq 2\exp\{\theta(\alpha )N-3\alpha \varepsilon N\}\geq1.
\end{equation}
By the Paley-Zygmund inequality,
\begin{equation}\label{InPlowerbound}
\P_0\Big[Z_{N,\alpha N-1,\varepsilon}\geq\exp\{\theta(\alpha )N-3\alpha \varepsilon N\}\Big]\geq \frac{\E_0\Big[Z_{N,\alpha N,\varepsilon}\Big]^2}{4\E_0\Big[Z^2_{N,\alpha N,\varepsilon}\Big]}.
\end{equation}

Let us bound $\E_0\Big[Z^2_{N,\alpha N,\varepsilon}\Big]$, which is equal to:
\begin{equation}\label{InPnewsecondmom}
\begin{split}
&\E_0\Big[\sum_{|\sigma|=|\sigma^\prime|=k}1_{(\sigma,\sigma^\prime\in\mathcal{A}_{N,k,\varepsilon})}\Big]=\E_0\Big[Z_{N, k,\varepsilon}\Big]+\E_0\Big[\sum_{q=0}^{k-1}\sum_{|\sigma\wedge\sigma^\prime|=q}1_{(\sigma,\sigma^\prime\in\mathcal{A}_{N,k,\varepsilon})}\Big]\\
=&\E_0\Big[Z_{N, k,\varepsilon}\Big]+\sum_{q=0}^{k-1}N^q N(N-1)N^{2k-2q-2}\P_0\Big(\sigma,\sigma^\prime\in\mathcal{A}_{N,k,\varepsilon}\Big\vert |\sigma\wedge\sigma^\prime|=q\Big),
\end{split}
\end{equation}
where $\sigma\wedge\sigma^\prime$ denotes the latest common ancestor of $\sigma$ and $\sigma^\prime$.

Recall that $\mathcal{A}_{N,k,\varepsilon}=\{\sigma\in\mathcal{A}_{N,k}; x_{\sigma_i}\geq \varepsilon+(1-\varepsilon)\frac{i-1}{k}, \forall 1\leq i\leq k\}$. $\P_0\Big(\sigma,\sigma^\prime\in\mathcal{A}_{N,k,\varepsilon}\Big\vert |\sigma\wedge\sigma^\prime|=q\Big)$ is hence equal to
\begin{equation}\label{InPinterprobab}
\begin{split}
&\int_{\varepsilon+(1-\varepsilon)(q-1)/k}^1\P_0\Big(\sigma,\sigma^\prime\in\mathcal{A}_{N,k,\varepsilon}\Big\vert |\sigma\wedge\sigma^\prime|=q, x_{\sigma\wedge\sigma^\prime}=y\Big)dy\\
=&\int_{\varepsilon+(1-\varepsilon)(q-1)/k}^1\P\Big(U_1<\cdots<U_{q-1}<y; U_i\geq\varepsilon+(1-\varepsilon)\frac{i-1}{k},\forall 1\leq i<q\Big)\\
&\quad\times\Big[\P\Big(y<U_{q+1}<\cdots<U_k; U_i\geq\varepsilon+(1-\varepsilon)\frac{i-1}{k},\forall q<i\leq k\Big)\Big]^2dy.
\end{split}
\end{equation}
Observe that
\begin{equation*}
\begin{split}
&\P\Big(y<U_{q+1}<\cdots<U_k; U_i\geq\varepsilon+(1-\varepsilon)\frac{i-1}{k},\forall q<i\leq k\Big)\\
\leq&\P\Big(U_{q+1}<\cdots<U_k; U_i\geq\varepsilon+(1-\varepsilon)\frac{i-1}{k},\forall q<i\leq k\Big)=\psi\Big(k-q, k-q, \varepsilon+(1-\varepsilon)\frac{q}{k}\Big).
\end{split}
\end{equation*}
Plugging it into (\ref{InPinterprobab}) implies that $\P_\varepsilon\Big(\sigma,\sigma^\prime\in\mathcal{A}_{N,k,\varepsilon}\Big\vert |\sigma\wedge\sigma^\prime|=q\Big)$ is less than
\begin{equation}\label{InPtwobranches}
\begin{split}
&\bigg\{\int_{\varepsilon+(1-\varepsilon)(q-1)/k}^1dy\P\Big(U_1<\cdots<U_{q-1}<y; U_i\geq\varepsilon+(1-\varepsilon)\frac{i-1}{k},\forall 1\leq i<q\Big)\times\\
&\P\Big(y<U_{q+1}<\cdots<U_k; U_i\geq\varepsilon+(1-\varepsilon)\frac{i-1}{k},\forall q<i\leq k\Big) \bigg\}\times \psi\Big(k-q, k-q, \varepsilon+(1-\varepsilon)\frac{q}{k}\Big)\\
&=\psi(k,k,\varepsilon)\times\psi\Big(k-q, k-q, \varepsilon+(1-\varepsilon)\frac{q}{k}\Big).
\end{split}
\end{equation}
Combining (\ref{InPnewsecondmom}) with (\ref{InPtwobranches}) yields that
\begin{equation}\label{InPsecondmomupp}
\begin{split}
\E_0\Big[Z_{N,k,\varepsilon}^2\Big]\leq&\E_0\Big[Z_{N, k,\varepsilon}\Big]+\frac{N-1}{N}\sum_{q=0}^{k-1}N^{2k-q}\psi(k,k,\varepsilon)\times\psi\Big(k-q, k-q, \varepsilon+(1-\varepsilon)\frac{q}{k}\Big)\\
=&\E_0\Big[Z_{N, k,\varepsilon}\Big]\Big(1+\frac{N-1}{N}\E_0\Big[Z_{N, k,\varepsilon}\Big]\sum_{q=0}^{k-1}N^{-q}\frac{\psi(k-q,k-q,\varepsilon+(1-\varepsilon)\frac{q}{k})}{\psi(k,k,\varepsilon)}\Big),
\end{split}
\end{equation}
where the last equality follows from (\ref{InPfirstmom}).
By (\ref{InPtypicalpath2}) and (\ref{InPeqstirling}),
\begin{equation}
\sum_{q=0}^{k-1}N^{-q}\frac{\psi(k-q,k-q,\varepsilon+(1-\varepsilon)\frac{q}{k})}{\psi(k,k,\varepsilon)}\leq\sum_{q=0}^{k-1}c_{5}\Big(\frac{k}{k-q}\Big)^{3/2}\Big(\frac{k}{e(1-\varepsilon)N}\Big)^{q}.
\end{equation}
For $k=\alpha N-1$ and $\alpha<e(1-\varepsilon)$, we get that for $N$ large enough,
\begin{equation*}
\sum_{q=0}^{k-1}c_{5}\Big(\frac{k}{k-q}\Big)^{3/2}\Big(\frac{k}{e(1-\varepsilon)N}\Big)^{q}\leq \sum_{q=0}^{k/2} c_6 \Big(\frac{\alpha}{e(1-\varepsilon)}\Big)^{q}+\sum_{q\geq k/2}c_6 q^{3/2}\Big(\frac{\alpha}{e(1-\varepsilon)}\Big)^{q}\leq c_7<\infty.
\end{equation*}
By (\ref{eq:lowerbound1mom}), for $N$ large enough, $\E_0\Big[Z_{N, k,\varepsilon}\Big]\geq1$. Going back to (\ref{InPsecondmomupp}), we obtain that for all $N$ sufficiently large,
\begin{equation}
\E_0\Big[Z_{N,\alpha N-1,\varepsilon}^2\Big]\leq(1+c_7)\E_0\Big[Z_{N, \alpha N-1,\varepsilon}\Big]^2.
\end{equation}
It then follows from (\ref{InPlowerbound}) that
\begin{equation}\label{InPmedialowerbound}
\P_0\Big[Z_{N,\alpha N-1,\varepsilon}\geq\exp\{\theta(\alpha )N-3\alpha \varepsilon N\}\Big]\geq \frac{1}{4(1+c_7)}=:c_8\in(0,1).
\end{equation}

For any vertex $\omega$ in the first generation, define $\mathcal{A}_{N,k+1,\varepsilon}(\omega)$ as follows:
\begin{equation*}
\mathcal{A}_{N,k+1,\varepsilon}(\omega):=\{|\sigma|=k+1; \sigma_1=\omega; x_{\sigma_2}<\cdots<x_{\sigma}; x_{\sigma_i}\geq\varepsilon+(1-\varepsilon)\frac{i-2}{k}, 2\leq i\leq k+1\}.
\end{equation*}
To bound $\P_0\{Z_{N,\alpha N}<\exp\{\theta(\alpha )N-3\alpha \varepsilon N\}\}$, we observe that
\begin{equation}
Z_{N,\alpha N}\geq \sum_{|\omega|=1}1_{(x_\omega<\varepsilon)}\sum_{|\sigma|=\alpha N}1_{(\sigma\in\mathcal{A}_{N,\alpha N,\varepsilon}(\omega))},
\end{equation}
where $\Big(x_\omega,\ \sum_{|\sigma|=\alpha N}1_{(\sigma\in\mathcal{A}_{N,\alpha N,\varepsilon}(\omega))}\Big)$ are i.i.d. Thus,
\begin{eqnarray}\label{InPupperboundprob}
&&\P_0\Big(Z_{N,\alpha N}<\exp\{\theta(\alpha )N-3\alpha \varepsilon N\}\Big)\nonumber\\
&\leq &\P_0\Big(\sum_{|\omega|=1}1_{(x_\omega<\varepsilon)}\sum_{|\sigma|=\alpha N}1_{(\sigma\in\mathcal{A}_{N,\alpha N,\varepsilon}(\omega))}<\exp\{\theta(\alpha )N-3\alpha \varepsilon N\}\Big)\nonumber\\
&\leq&\P_0\Big(1_{(x_\omega<\varepsilon)}\sum_{|\sigma|=\alpha N}1_{(\sigma\in\mathcal{A}_{N,\alpha N,\varepsilon}(\omega))}<\exp\{\theta(\alpha )N-3\alpha \varepsilon N\}\Big)^N.
\end{eqnarray}
The fact that $\P_0[\sigma\in\mathcal{A}_{N,k+1,\varepsilon}(\omega)\vert x_\omega<\varepsilon]=\P_0[\sigma\in\mathcal{A}_{N,k,\varepsilon}]$ implies that given $\{x_\omega<\varepsilon\}$, $\sum_{|\sigma|=\alpha N}1_{(\sigma\in\mathcal{A}_{N,\alpha N,\varepsilon}(\omega))}$ is distributed as $Z_{N,\alpha N-1,\varepsilon}$ under $\P_0$. Therefore, we have
\begin{eqnarray*}
&&\P_0\Big(1_{(x_\omega<\varepsilon)}\sum_{|\sigma|=\alpha N}1_{(\sigma\in\mathcal{A}_{N,\alpha N,\varepsilon}(\omega))}<\exp\{\theta(\alpha )N-3\alpha \varepsilon N\}\Big)\\
&\leq& 1-\varepsilon+\varepsilon\P_0\bigg(\sum_{|\sigma|=\alpha N}1_{(\sigma\in\mathcal{A}_{N,\alpha N,\varepsilon}(\omega))}<\exp\{\theta(\alpha )N-3\alpha \varepsilon N\}\Big\vert x_\omega<\varepsilon\bigg)\\
&=&1-\varepsilon+\varepsilon\Big(1-\P_0\Big[Z_{N,\alpha N-1,\varepsilon}\geq\exp\{\theta(\alpha )N-3\alpha \varepsilon N\}\Big]\Big),
\end{eqnarray*}
which is bounded by $1-\varepsilon+\varepsilon(1-c_8)$ because of (\ref{InPmedialowerbound}). Plugging this inequality into (\ref{InPupperboundprob}) yields that
\begin{eqnarray}
\P_0\Big[Z_{N,\alpha N}<\exp\{\theta(\alpha )N-3\alpha \varepsilon N\}\Big]&\leq&\Big(1-\varepsilon+\varepsilon(1-c_8)\Big)^N \nonumber\leq e^{-Nc_8\varepsilon},
\end{eqnarray}
which is summable in $N$. By the Borel-Cantelli lemma, we conclude that for $\varepsilon$ sufficiently small, $\P_0-$almost surely,
\begin{equation}
\liminf_{N\rightarrow\infty}\frac{\log Z_{N,\alpha N}}{N}\geq\theta(\alpha)-3\alpha\varepsilon,
\end{equation}
completing the proof of (i) of Theorem \ref{InPlimitcurve}.

\bigskip
Before the proof of Part (ii), we turn to estimate $\P_0\Big[Z_{N,\alpha N}\geq1\Big]$ with $\alpha>e$.

\noindent\textit{Proof of (iii) of Theorem \ref{InPlimitcurve}.}
The upper bound is easy. By Markov's inequality and (\ref{eq:1mom}),
\begin{equation*}
\P_0\Big[Z_{N,\alpha N}\geq1\Big]\leq\E_0\Big[Z_{N,\alpha N}\Big]\leq \frac{e^{\theta(\alpha)N}}{2\sqrt{\alpha N}}.
\end{equation*}
It follows that
\begin{equation}\label{InPsubcriticalupper}
\limsup_{N\rightarrow\infty}\frac{\log \P_0\Big[Z_{N,\alpha N}\geq1\Big]}{N}\leq \theta(\alpha)<0.
\end{equation}
To get the lower bound, we use the fact that $Z_{N,k}\geq Z_{N,k,\varepsilon}$ and the Paley-Zygmund inequality to get that for any $\varepsilon\in[0,1)$,
\begin{equation}\label{InPPaley-Zygmund}
\P_0\Big[Z_{N,k}\geq1\Big]\geq\P_0\Big[Z_{N,k,\varepsilon}\geq1\Big]\geq\frac{\E_0\Big[Z_{N,k,\varepsilon}\Big]^2}{\E_0\Big[Z_{N,k,\varepsilon}^2\Big]}.
\end{equation}
In this part, we always take $\varepsilon=0$. Applying (\ref{InPfirstmom}) and Stirling's formula (\ref{InPeqstirling}) gives that for $k=\alpha N$,
\begin{equation}\label{InPfirstmomlower}
\E_0\Big[Z_{N, k,0}\Big]\geq \frac{N^{k}}{(k+1)!}\geq \frac{e^{\theta(\alpha )N}}{3(\alpha N+1)^{3/2}}.
\end{equation}

On the other hand, in view of (\ref{InPsecondmomupp}), we obtain that
\begin{equation}\label{InPsecondmomupp0}
\E_0\Big[Z_{N,k,0}^2\Big]\leq\E_0\Big[Z_{N, k,0}\Big]\Big(1+\frac{N-1}{N}\sum_{q=0}^{k-1}N^{k-q}\psi(k-q,k-q,\frac{q}{k})\Big).
\end{equation}
By (\ref{InPtypicalpath2}) and (\ref{InPeqstirling}), one sees that for $k=\alpha N$ with $\alpha >e$,
\begin{eqnarray}
\sum_{q=0}^{k-1}N^{k-q}\psi(k-q,k-q,\frac{q}{k})&=&\sum_{q=0}^{k-1}N^{k-q}\frac{(1+1/(k-q))^{k-q}}{(k-q+1)!}(1-q/k)^{k-q}\nonumber\\
&\leq&\sum_{q=0}^{k-1}\frac{e}{2(k-q)^{3/2}}\Big(\frac{e N}{k}\Big)^{k-q}\leq \sum_{q=0}^{k-1}\frac{e}{2(k-q)^{3/2}},\nonumber
\end{eqnarray}
Let $c_9:=\sum_{q=1}^\infty\frac{e}{2q^{3/2}}\in(0,\infty)$. It follows that
\begin{equation}
\sum_{q=0}^{k-1}N^{k-q}\psi(k-q,k-q,\frac{q}{k})\leq c_9.
\end{equation}
Plugging it into (\ref{InPsecondmomupp0}) shows that
\begin{equation}\label{InPsecondmomupp1}
\E_0\Big[Z_{N,\alpha N,0}^2\Big]\leq\E_0\Big[Z_{N, \alpha N,0}\Big](1+c_9).
\end{equation}
According to (\ref{InPPaley-Zygmund}) and (\ref{InPfirstmomlower}), we obtain that
\begin{equation}\label{InPlowerbdofsurvie}
\P_0\Big[Z_{N,\alpha N}\geq1\Big]\geq\frac{\E_0\Big[Z_{N, \alpha N,0}\Big]}{1+c_9}\geq c_{10}\frac{e^{\theta(\alpha )N}}{(\alpha N+1)^{3/2}},
\end{equation}
where $c_{10}:=\frac{1}{3(1+c_9)}$. Therefore, we conclude that for $\alpha>e$,
\begin{equation}
\liminf_{N\rightarrow\infty}\frac{\log \P_0\Big[Z_{N,\alpha N}\geq1\Big]}{N}\geq \theta(\alpha),
\end{equation}
which completes the proof of (iii) of Theorem \ref{InPlimitcurve}.

\bigskip

\noindent\textit{Proof of (ii) of Theorem \ref{InPlimitcurve}}. Let us estimate $\P_0[Z_{N,e N}\geq1]$.

For the lower bound, one observes that the inequality (\ref{InPlowerbdofsurvie}) still holds when $\alpha=e$. As $\theta(e)=0$, we get that
\begin{equation}
\P_0[Z_{N,e N}\geq1]\geq c_{11}N^{-3/2}.
\end{equation}

To obtain the upper bound, we introduce the following collections of accessible vertices in $T^{(N)}$:
\begin{equation}
\mathcal{A}_L(K):=\{|\sigma|=K: x_{\sigma_1}< \cdots< x_{\sigma}; x_{\sigma_{j}}\geq\frac{(j-L)_+}{K+1}; \forall 1\leq j\leq K\},\quad 0\leq L<K.
\end{equation}

Set $K=e N$ and $L_0=2\log K$. One observes that
\begin{equation}
\mathcal{A}_{N,K}\subset\mathcal{A}_{L_0}(K)\cup\bigcup_{k=L_0+1}^K\{\exists |\sigma|=k:x_{\sigma_1}< \cdots< x_{\sigma}, x_\sigma<\frac{k-L_0}{K+1}\}.
\end{equation}
As a consequence,
\begin{eqnarray}\label{InPtypicalsurvival}
\P_0\Big[Z_{N,e N}\geq1\Big]&\leq& \P_0\Big[\exists \sigma\in\mathcal{A}_{L_0}(K)\Big]+\sum_{k=L_0+1}^K\P_0\Big[\exists |\sigma|=k:x_{\sigma_1}< \cdots< x_{\sigma}, x_\sigma<\frac{k-L_0}{K+1} \Big]\nonumber\\
&\leq&\E_0\Big[\sum_{|\sigma|=K}1_{(\sigma\in \mathcal{A}_{L_0}(K))}\Big]+\sum_{k=L_0+1}^K \E_0\Big[\sum_{\sigma\in\mathcal{A}_{N,k}}1_{(x_\sigma<\frac{k-L_0}{K+1})}\Big],
\end{eqnarray}
where the last inequality follows from Markov's inequality. We first compute the second term on the right-hand side of (\ref{InPtypicalsurvival}), which is
\begin{eqnarray}
\sum_{k=L_0+1}^K \E_0\Big[\sum_{\sigma\in\mathcal{A}_{N,k}}1_{(x_\sigma<\frac{k-L_0}{K+1})}\Big]&=&\sum_{k=L_0+1}^K N^k\P\Big[U_1<\cdots<U_k<\frac{k-L_0}{K+1}\Big]\nonumber\\
&=&\sum_{k=L_0+1}^K N^k\Big(\frac{k-L_0}{K+1}\Big)^k\frac{1}{k!}.
\end{eqnarray}
By (\ref{InPeqstirling}),
\begin{equation}
\sum_{k=L_0+1}^K \E_0\Big[\sum_{\sigma\in\mathcal{A}_{N,k}}1_{(x_\sigma<\frac{k-L_0}{K+1})}\Big]\leq\sum_{k=L_0+1}^K \Big(\frac{e N}{K+1}\Big)^k\frac{e^{-L_0}}{2\sqrt{k}}\leq c_{12}N^{-3/2}.
\end{equation}
The inequality (\ref{InPtypicalsurvival}) thus becomes that
\begin{eqnarray}
\P_0\Big[Z_{N,e N}\geq1\Big]&\leq& \E_0\Big[\sum_{|\sigma|=K}1_{(\sigma\in \mathcal{A}_{L_0}(K))}\Big]+c_{12}N^{-3/2}\nonumber\\
&=& N^K \P[A_{L_0}(K)]+c_{12}N^{-3/2},
\end{eqnarray}
where $A_{L_0}(K)$ is defined in (\ref{InPdefofA}). Applying Lemma \ref{InPineqofA} yields that
\begin{eqnarray}
\P_0\Big[Z_{N,e N}\geq1\Big]&\leq& N^K\frac{e^{c_0\sqrt{L_0}}}{K^{3/2}}\frac{e^K}{(K+1)^K}+c_{12}N^{-3/2}\nonumber\\
&\leq & c_{13}\frac{e^{c_0\sqrt{2\log K}}}{N^{3/2}}=N^{-3/2+o_N(1)},
\end{eqnarray}
which completes the proof of (ii) of Theorem \ref{InPlimitcurve}.
\end{proof}

\section{The criticality at $eN-\frac{3}{2}\log N$}
\label{ss:critical}
In this section, we prove Theorem \ref{thm:critical} and Proposition \ref{prop:critical}, which says that
\begin{equation}
\lim_{N\rightarrow\infty}\P_0\Big(Z_{N,eN-\beta \log N}\geq1\Big)=\left\{\begin{array}{cl}
1 & \textrm{if }\beta >3/2;\\
0 & \textrm{if }\beta <3/2.
\end{array}\right.
\end{equation}
and that when $\beta=3/2$, for any $\varepsilon>0$ and $N$ sufficiently large,
\begin{equation}
\P_0\Big(Z_{N,eN-\beta \log N}\geq1\Big)\geq N^{-\varepsilon}.
\end{equation}

\subsection{Extinction after $eN-\beta\log N$ for any $\beta<3/2$}
Let $K=eN-\beta\log N$ with $\beta<3/2$ fixed. Similarly as (\ref{InPtypicalsurvival}), one sees that
\begin{equation}
\label{eq:ext}
\begin{split}
\P_0\Big[Z_{N,K}\geq1\Big]\leq& \P_0\Big[\exists \sigma\in\mathcal{A}_{L_0}(K)\Big]+\sum_{k=L_0+1}^K\P_0\Big[\exists |\sigma|=k:x_{\sigma_1}< \cdots< x_{\sigma}, x_\sigma<\frac{k-L_0}{K+1} \Big]\\
\leq&\E_0\Big[\sum_{|\sigma|=K}1_{(\sigma\in \mathcal{A}_{L_0}(K))}\Big]+\sum_{k=L_0+1}^K \E_0\Big[\sum_{\sigma\in\mathcal{A}_{N,k}}1_{(x_\sigma<\frac{k-L_0}{K+1})}\Big]\\
\leq&N^K \P[A_{L_0}(K)]+\sum_{k=L_0+1}^K \Big(\frac{e N}{K+1}\Big)^k\frac{e^{-L_0}}{2\sqrt{k}}.
\end{split}
\end{equation}
We take $L_0=2\log N$. Note that for $1\leq k\leq K$, 
\begin{equation}
\Big(\frac{e N}{K+1}\Big)^k\frac{e^{-L_0}}{2\sqrt{k}}\leq \Big(1+\frac{\beta\log N}{K}\Big)^K\frac{N^{-2}}{2\sqrt{k}}\leq N^{\beta-2}\frac{1}{2\sqrt{k}}.
\end{equation}
So, the second sum on the right-hand side of (\ref{eq:ext}) is less than
\begin{equation}
\sum_{k=L_0+1}^KN^{\beta-2}\frac{1}{2\sqrt{k}}\leq c_{14}N^{\beta-3/2},
\end{equation}
which converges to zero if $\beta<3/2$.

Applying Lemma \ref{InPineqofA} for $A_{L_0}(K)$ yields that
\begin{equation}
N^K \P[A_{L_0}(K)]\leq \frac{e^{c_0\sqrt{L_0}}}{K^{3/2}}\Big(\frac{eN}{K+1}\Big)^K\leq \frac{e^{c_0\sqrt{L_0}}}{K^{3/2}}\Big(1+\frac{\beta\log N}{K}\Big)^K\leq c_{15}N^{\beta-3/2}e^{c_0\sqrt{L_0}},
\end{equation}
which also converges to zero as $N\rightarrow\infty$. 

Consequently, when $\beta<3/2$,
\begin{equation}
\lim_{N\rightarrow\infty}\P_0\Big(Z_{N,eN-\beta \log N}\geq1\Big)=0.
\end{equation}

\subsection{Survival until $eN-\beta\log N$ for any $\beta>3/2$}
It remains to show that $\lim_{N\rightarrow\infty}\P_0\Big(Z_{N,eN-\beta \log N}\geq1\Big)=1$ when $\beta>3/2$.

Let $k_0:=\gamma\log N$, $K_0:=eN-(\beta+\gamma)\log N$ and $\delta_N:=\frac{(\gamma +3/2)\log N}{eN}$. We define $\widetilde{\mathcal{A}}(k_0,\delta_N,K_0)$ to be the collection of accessible individuals satisfying that
\begin{equation}
\begin{split}
&0<x_{\sigma_1}<\cdots<x_{\sigma_{k_0}}\leq \delta_N;\\
&\delta_N<x_{\sigma_{k_0+1}}<\cdots<x_{\sigma_{K_0+k_0}}=x_{\sigma}\leq 1
\textrm{ and } x_{\sigma_{k_0+j}}\geq \delta_N+\Big(1-\delta_N\Big)\frac{j-1}{K_0}.
\end{split}
\end{equation}
Clearly, $\P_0\Big(Z_{N,eN-\beta \log N}\geq1\Big)\geq\P_0\Big(\#\widetilde{\mathcal{A}}(k_0,\delta_N,K_0)\geq1\Big)$. Recalling the definitions of $Z_{N, K_0,\delta_N}$ and $Z_{N,k_0}(\delta_N)$ in (\ref{InPbar}) and (\ref{InPpartialpopulation}) respectively, one observes that 
\begin{equation}
\P_0\Big(Z_{N, eN-\beta \log N}=0\Big)\leq \E\Bigg(\bigg\{1-\P_0\Big(Z_{N, K_0,\delta_N}\geq1\Big)\bigg\}^{Z_{N,k_0}(\delta_N)}\Bigg).
\end{equation}
We first give a lower bound for the survival probability $\P_0\Big(Z_{N, K,\delta}\geq1\Big)$. It follows from the Paley-Zygmund inequality that
\begin{equation}
\label{eq:paley-zyg}
\P_0\Big(Z_{N, K,\delta}\geq1\Big)\geq \frac{\E_0\Big(Z_{N,K,\delta}\Big)^2}{\E_0\Big(Z^2_{N,K,\delta}\Big)},
\end{equation}
where the first moment of $Z_{N,K,\delta}$ is as follows:
\begin{equation}
\E_0\Big(Z_{N,K,\delta}\Big)=N^K\psi(K,K,\delta)=N^{K}\frac{(1+1/K)^K}{(K+1)!}(1-\delta)^K.
\end{equation}
By (\ref{InPsecondmomupp}) again
\begin{equation}
\begin{split}
\E_0\Big[Z_{N,K,\delta}^2\Big]
&\leq\E_0\Big[Z_{N, K,\delta}\Big]\Big\{1+\E_0\Big[Z_{N, K,\delta}\Big] \sum_{q=0}^{K-1}c_{5}\Big(\frac{K}{K-q}\Big)^{3/2}\Big(\frac{K}{e(1-\delta)N}\Big)^{q}\Big\}.
\end{split}
\end{equation}
Here we take $\delta=\delta_N$ and $K=K_0$. On the one hand,
\begin{equation}
c_{16}N^{-3/2}\leq \E_0\Big(Z_{N,K_0,\delta_N}\Big)\leq c_{17} N^{-3/2}.
\end{equation}
On the other hand, as $\sum_{q=0}^{K_0-1}c_{5}\Big(\frac{K_0}{K_0-q}\Big)^{3/2}\Big(\frac{K_0}{e(1-\delta_N)N}\Big)^{q}\leq c_{18}N^{3/2}$,
\begin{equation}
\E_0\Big[Z_{N,K_0,\delta_N}^2\Big]\leq c_{19}\E_0\Big[Z_{N, K_0,\delta_N}\Big] .
\end{equation}
As a consequence,
\begin{equation}\label{eq:lowerbdofsur}
\P_0\Big(Z_{N, K_0,\delta_N}\geq1\Big)\geq c_{20} \E_0\Big[Z_{N, K_0,\delta_N}\Big]\geq c_{21} N^{-3/2}.
\end{equation}
We deduce that 
\begin{equation}\label{eq:exatbeta}
\P_0\Big(Z_{N, eN-\beta \log N}=0\Big)\leq \E\Bigg(\bigg\{1-c_{21} N^{-3/2}\bigg\}^{Z_{N,k_0}(\delta_N)}\Bigg).
\end{equation}

We are going to prove that with high probability $Z_{N,k_0}(\delta_N)\gg N^{3/2}$. Take $\varepsilon>0$ sufficiently small so that $\beta-3\varepsilon>3/2$. Let $\beta^\prime=\beta-\varepsilon>0$, $\varepsilon_N:=\frac{\varepsilon\log N}{eN}$ and $\delta^\prime_N:=\frac{(\beta^\prime+\gamma)\log N}{eN}$. It immediately follows that
\begin{equation*}
\begin{split}
&Z_{N,k_0}(\delta_N)=\sum_{|\sigma|=k_0}1_{\{0<x_{\sigma_1}<\cdots<x_{\sigma_{k_0}}\leq \delta_N\}}\\
\geq& \#\left\{|\sigma|=k_0: 0<x_{\sigma_1}\leq\varepsilon_N<x_{\sigma_2}<\cdots<x_{\sigma_{k_0}}\leq \varepsilon_N+\delta^\prime_N; x_{\sigma_j}\geq \varepsilon_N+\frac{j-2}{k_0-1}\delta^\prime_N,\forall j\geq2\right\}\\
=&\sum_{|\omega|=1}1_{\{x_\omega\leq\varepsilon_N\}}\sum_{|\sigma|=k_0, \sigma>\omega}1_{\{\varepsilon_N<x_{\sigma_2}<\cdots<x_{\sigma_{k_0}}\leq \varepsilon_N+\delta^\prime_N; x_{\sigma_j}\geq \varepsilon_N+\frac{j-2}{k_0-1}\delta^\prime_N,\forall j\geq2\}},
\end{split}
\end{equation*}
where $\sum_{|\sigma|=k_0, \sigma>\omega}1_{\{\varepsilon_N<x_{\sigma_2}<\cdots<x_{\sigma_{k_0}}\leq \varepsilon_N+\delta^\prime_N; x_{\sigma_j}\geq \varepsilon_N+\frac{j-2}{k_0-1}\delta^\prime_N,\forall j\geq2\}}$ is distributed as $Z_{N,k_0-1,1-\delta_N^\prime}$ under $\P_0$. This implies that
\begin{equation}\label{eq:nbofdelta}
\begin{split}
&\P_0\Big(Z_{N,k_0}(\delta_N)\leq N^{\frac{3+\varepsilon}{2}}\Big)\\
\leq& \Big[\P\Big(1_{\{x_\omega\leq\varepsilon_N\}}\sum_{|\sigma|=k_0, \sigma>\omega}1_{\{\varepsilon_N<x_{\sigma_2}<\cdots<x_{\sigma_{k_0}}\leq \varepsilon_N+\delta^\prime_N; x_{\sigma_j}\geq \varepsilon_N+\frac{j-2}{k_0-1}\delta^\prime_N,\forall j\geq2\}}\leq N^{\frac{3+\varepsilon}{2}}\Big)\Big]^N\\
=&\Big(1-\varepsilon_N+\varepsilon_N\P_{0}\Big(Z_{N,k_0-1,1-\delta_N^\prime}\leq N^{\frac{3+\varepsilon}{2}} \Big)\Big)^N.
\end{split}
\end{equation}

Recall that
\begin{equation}
\E_{0}\Big(Z_{N,k-1,1-\delta^\prime}\Big)= N^{k-1}\psi(k-1,k-1,1-\delta^\prime).
\end{equation}
Then there exist two constants $c_{\pm}(\beta,\gamma)$ such that
\begin{equation}
\frac{c_-(\beta,\gamma)N^{\gamma\log(1+\frac{\beta-\varepsilon}{\gamma})}}{(\log N)^{3/2}}\leq \E_{0}\Big(Z_{N,k_0-1,1-\delta^\prime_N}\Big)\leq \frac{c_+(\beta,\gamma)N^{\gamma\log(1+\frac{\beta-\varepsilon}{\gamma})}}{(\log N)^{3/2}}.
\end{equation}
As $\gamma$ goes to infinity, $\gamma\log(1+\frac{\beta-\varepsilon}{\gamma})\rightarrow \beta^\prime>3/2+2\varepsilon$. Take $\gamma$ sufficiently large so that $\gamma\log(1+\frac{\beta-\varepsilon}{\gamma})>3/2+\varepsilon$. For all $N$ sufficiently large, we have $\E_{0}\Big(Z_{N,k_0-1,1-\delta^\prime_N}\Big)\geq 2N^{\frac{3+\varepsilon}{2}}$. 

By (\ref{InPsecondmomupp}), there exists a constant $C(\beta,\gamma)>0$ such that
\begin{equation*}
\begin{split}
\E_{0}\Big(Z^2_{N,k_0-1,1-\delta^\prime_N}\Big)&\leq\E_{0}\Big(Z_{N,k_0-1,1-\delta^\prime_N}\Big)\Big\{1+\E_{0}\Big(Z_{N,k_0-1,1-\delta^\prime_N}\Big)\sum_{q=0}^{k_0-2}c_{5}\Big(\frac{k_0-1}{k_0-1-q}\Big)^{3/2}\Big(\frac{k_0-1}{eN\delta^\prime_N}\Big)^{q}\Big\}\\
&\leq C(\beta,\gamma)\E_{0}\Big(Z_{N,k-1,1-\delta^\prime}\Big)^2.
\end{split}
\end{equation*}
By the Paley-Zygmund inequality,
\begin{equation*}
\P_{0}\Big(Z_{N,k_0-1,1-\delta^\prime_N}\geq\frac{1}{2}\E_{0}\Big(Z_{N,k_0-1,1-\delta^\prime_N}\Big)\Big)\geq\P_{0}\Big(Z_{N,k-1,1-\delta^\prime}\geq N^{\frac{3+\varepsilon}{2}}\Big)\geq \frac{1}{4C(\beta,\gamma)}>0.
\end{equation*}
Plugging it into (\ref{eq:nbofdelta}) implies that
\begin{equation}
\begin{split}
\P_0\Big(Z_{N,k_0}(\delta_N)\leq N^{\frac{3+\varepsilon}{2}}\Big)&\leq \Bigg(1-\varepsilon_N+\varepsilon_N\Big[1-\P_{0}\Big(Z_{N,k_0-1,1-\delta_N^\prime}\geq N^{\frac{3+\varepsilon}{2}} \Big)\Big]\Bigg)^N\\
&\leq e^{-N\varepsilon_N\P_{0}\Big(Z_{N,k_0-1,1-\delta_N^\prime}\geq N^{\frac{3+\varepsilon}{2}} \Big)}\leq e^{-\varepsilon c_{22}\log N}\rightarrow 0.
\end{split}
\end{equation}
It follows from (\ref{eq:exatbeta}) that
\begin{equation}
\begin{split}
&\P_0\Big(Z_{N, eN-\beta \log N}=0\Big)\\
\leq& \E\Bigg(\bigg\{1-c_{21} N^{-3/2}\bigg\}^{Z_{N,k_0}(\delta_N)}; Z_{N,k_0}(\delta_N)\geq N^{\frac{3+\varepsilon}{2}}\Bigg)+\P_0\Big(Z_{N,k_0}(\delta_N)\leq N^{\frac{3+\varepsilon}{2}}\Big)\\
\leq& (1-c_{21}N^{-3/2})^{N^{\frac{3+\varepsilon}{2}}}+e^{-\varepsilon c_{22}\log N}
\rightarrow 0.
\end{split}
\end{equation}
This tells us that $\lim_{N\rightarrow\infty}\P_0\Big(Z_{N, eN-\beta \log N}\geq0\Big)=1$ with $\beta>3/2$.
\subsection{Proof of Proposition \ref{prop:critical}: $\beta=3/2$}

In this subsection, we consider the probability $\P_0\Big(Z_{N, eN-3/2\log N}\geq1\Big)$. Recounting the arguments in the previous subsection with $k_0=\gamma\log N$, $K_0=eN-(3/2+\gamma)\log N$ and $\delta_N=\frac{(\gamma +3/2)\log N}{eN}$. Again, 
\begin{equation*}
\P_0\Big(Z_{N, eN-3/2\log N}\geq1\Big)\geq\P_0\Big(\#\widetilde{\mathcal{A}}(k_0,\delta_N,K_0)\geq1\Big).
\end{equation*}
Recall that $\#\widetilde{\mathcal{A}}(k_0,\delta_N,K_0)$ is equal to 
\begin{equation}
\sum_{|\omega|=k_0}1_{\{x_{\omega_1}<\cdots<x_{\omega_k}\leq \delta_N\}}\sum_{|\sigma|=K_0+k_0; \sigma>\omega}1_{\{\delta_N<x_{\sigma_{k_0+1}}<\cdots<x_{\sigma_{K_0+k_0}}=x_{\sigma}\leq 1; x_{\sigma_{k_0+j}}\geq \delta_N+(1-\delta_N)\frac{j-1}{K_0},\forall j\geq1\}},
\end{equation}
where $\sum_{|\sigma|=K_0+k_0; \sigma>\omega}1_{\{\delta_N<x_{\sigma_{k_0+1}}<\cdots<x_{\sigma_{K_0+k_0}}=x_{\sigma}\leq 1; x_{\sigma_{k_0+j}}\geq \delta_N+(1-\delta_N)\frac{j-1}{K_0},\forall j\geq1\}}$ is distributed as $Z_{N,K_0,\delta_N}$.

One hence sees that
\begin{equation*}
\begin{split}
&\P_0\Big(\#\widetilde{\mathcal{A}}(k_0,\delta_N,K_0)\geq1\Big)\\
\geq&\P_0\Big(\sum_{|\omega|=k}1_{\{x_{\omega_1}<\cdots<x_{\omega_k}\leq \delta_N\}}\geq N^{3/2-\varepsilon}\Big)\Big\{1-\Big(1-\P_{0}(Z_{N,K_0,\delta_N}\geq1)\Big)^{N^{3/2-\varepsilon}}\Big\}.
\end{split}
\end{equation*}
By (\ref{eq:lowerbdofsur}), $\P_{0}(Z_{N,K_0,\delta_N}\geq1)\geq c_{21}N^{-3/2}$. We get that
\begin{equation*}
\begin{split}
\P_0\Big(\#\widetilde{\mathcal{A}}(k,\delta_N,K)\geq1\Big)\geq&\P\Big(Z_{N,k_0}(\delta_N)\geq N^{3/2-\varepsilon}\Big)\Big\{1-\Big(1-c_{21}N^{-3/2}\Big)^{N^{3/2-\varepsilon}}\Big\}\\
\geq&\P_{0}\Big(Z_{N,k_0,1-\delta_N}\geq N^{3/2-\varepsilon}\Big)\Big\{1-\Big(1-c_{21}N^{-3/2}\Big)^{N^{3/2-\varepsilon}}\Big\}.
\end{split}
\end{equation*}
Similarly as above, there exist two constants $c_{\pm}(\gamma)$ such that
\begin{equation}
\frac{c_-(\gamma)N^{\gamma\log(1+\frac{3/2}{\gamma})}}{(\log N)^{3/2}}\leq \E_{0}\Big(Z_{N,k_0,1-\delta_N}\Big)\leq \frac{c_+(\gamma)N^{\gamma\log(1+\frac{3/2}{\gamma})}}{(\log N)^{3/2}}.
\end{equation}
There exists a constant $C(\gamma)>0$ such that
\begin{equation}
\E_{0}\Big(Z^2_{N,k_0,1-\delta_N}\Big)\leq C(\gamma)\E_{0}\Big(Z_{N,k_0,1-\delta_N}\Big)^2.
\end{equation}
As $\gamma\uparrow\infty$, $\gamma\log(1+\frac{3/2}{\gamma})\uparrow 3/2$. Take $\gamma>0$ large enough such that $\gamma\log(1+\frac{3/2}{\gamma})>3/2-\varepsilon$. By the Paley-Zygmund inequality, we obtain that
\begin{equation}
\P_{0}\Big(Z_{N,k_0,1-\delta_N}\geq N^{3/2-\varepsilon}\Big)\geq \frac{1}{4C(\gamma)}.
\end{equation}
We deduce that for all $N$ sufficiently large,
\begin{equation}
\P_0\Big(\#\widetilde{\mathcal{A}}(k_0,\delta_N,K_0)\geq1\Big)\geq\frac{1}{4C(\gamma)}\Big\{1-\Big(1-c_{21}N^{-3/2}\Big)^{N^{3/2-\varepsilon}}\Big\}\geq c_{23}N^{-\varepsilon}.
\end{equation}
We thus conclude that for any $\varepsilon>0$ and all $N$ large enough, 
\begin{equation}
\P_0\Big(Z_{N, eN-3/2\log N}\geq1\Big)\geq N^{-\varepsilon}.
\end{equation}

\appendix
\section{Coupling with a branching process}
By considering the typical increasing paths, it has been proven that the critical value for $\alpha$ is $\alpha_c=e$. In what follows, by coupling with a branching process, we give an auxiliary idea to show the following result.
\begin{equation}\label{eq:criticalate}
\lim_{N\rightarrow\infty}\P_0\Big(Z_{N,\alpha N}\geq1\Big)=1, \forall \alpha\in(0,e).
\end{equation}

In the same probability space, we introduce accessibility percolation on a Galton-Watson tree as follows. For $\Lambda>0$, let $\mathcal{T}^\Lambda$ be a Galton-Watson tree rooted also at $\varnothing$, whose offspring distribution is Poisson with parameter $\Lambda$. To each vertex $\xi\in\mathcal{T}^{\Lambda}\setminus\{\varnothing\}$, we attach an random variable $x_\xi$, which is independent of $x_\varnothing$. Assume that all these variables $x_\xi$, $\xi\in\mathcal{T}^\Lambda$ are i.i.d., following the law $U[0,1]$. Similarly, let $[\![\varnothing,\xi]\!]$ denote the ancestral line of $\xi$ in $\mathcal{T}^{\Lambda}$. We keep $\xi$ if the attached random variables along its ancestral line $[\![\varnothing,\xi]\!]$ is \textit{decreasing} and delete all other vertices. Let $D_k^{(\Lambda)}$ be the number of individuals alive at $k$-th generation. Let $d_{k}^{(\Lambda)}(s,x)$ denote the generating function of $D_k^{(\Lambda)}$ under $\P_x$. Similarly to (\ref{InPrecursioneq}), we get the following recursive equation.
\begin{equation}\label{InPrecursiveeqofd}
d_{k+1}^{(\Lambda)}(s,x)=\E_x\Big[s^{D_{k+1}^{(\Lambda)}}\Big]=\exp\Big\{-\Lambda x+\Lambda\int_0^x d_k^{(\Lambda)}(s,y)dy\Big\},\quad \forall k\geq1.
\end{equation}
In particular, $d_1^{(\Lambda)}(s,x)=\exp\{\Lambda x(s-1)\}$. We also note that $d_{k}^{(\Lambda)}(s,x)\leq d_{k}^{(\Lambda)}(s,y)$ if $x\geq y$.

We compare the generating functions $f_k^{(N)}$ and $d_k^{(\Lambda)}$ via the following lemma.

\begin{lemma}
For any $0< \Lambda\leq N$ and $u\in[0,1]$, we have
\begin{equation}\label{InPcouplinggenerating}
f_k^{(N)}(s,\frac{\Lambda}{N}u)\leq d_k^{(\Lambda)}(s,u),\quad\forall k\geq1.
\end{equation}
\end{lemma}

\begin{proof}[Proof.]
For $N\geq \Lambda>0$ and $u\in[0,1]$,
\begin{equation}
f_1^{(N)}(s, \frac{\Lambda}{N}u)=\Big(1-\frac{\Lambda}{N}u+\frac{\Lambda}{N}us\Big)^N\leq \exp\{\Lambda u(s-1)\}= d_1^{(\Lambda)}(s,u).
\end{equation}
Assume that $f_k^{(N)}(s,\frac{\Lambda}{N}u)\leq d_k^{(\Lambda)}(s,u)$ holds for $k\geq1$. Then,
\begin{eqnarray*}
f_{k+1}^{(N)}(s,\frac{\Lambda}{N}u)&=&\Big[1-\frac{\Lambda}{N}u+\int_0^{\Lambda u/N}f_k^{(N)}(s,y)dy\Big]^N\\
&=&\Big[1-\frac{\Lambda}{N}u+\frac{\Lambda}{N}\int_0^{u}f_k^{(N)}(s,\frac{\Lambda}{N}v)dv\Big]^N\\
&\leq& \exp\Big\{-\Lambda u+\Lambda \int_0^{u}f_k^{(N)}(s,\frac{\Lambda}{N}v)dv\Big\},
\end{eqnarray*}
which is bounded by $\exp\Big\{-\Lambda u+\Lambda\int_0^{u}d_k^{(\Lambda)}(s,v)dv\Big\}$.  It follows from (\ref{InPrecursiveeqofd}) that
\begin{equation}
f_{k+1}^{(N)}(s,\frac{\Lambda}{N}u)\leq d_{k+1}^{(\Lambda)}(s,u).
\end{equation}
Therefore, by induction on $k$, we have $f_k^{(N)}(s,\frac{\Lambda}{N}u)\leq d_k^{(\Lambda)}(s,u)$ for any $k\geq1$.
\end{proof}

With the help of this lemma, we show that with positive probability, there exists at least one accessible vertex at the $\alpha N$-th generation for $\alpha<e$.

\begin{lemma}\label{lem:InPuniformlowerbd}
Let $\alpha\in(0,e)$. For any $\delta\in(\frac{\alpha}{e}, 1\wedge\alpha)$, there exists some positive constant $c(\delta,\alpha)>0$ such that
\begin{equation}\label{InPuniformlowerbd}
\inf_{N\geq 1}\P\Big[Z_{N,\alpha N}(\delta)\geq1\Big]>c(\delta,\alpha).
\end{equation}
\end{lemma}

\begin{proof}[Proof.]
Set $K=\alpha N$. It follows from (\ref{InPcouplinggenerating}) that for $a\in\mathbb{N}_+$ and $a\leq N$,
\begin{equation}\label{InPcouplinggenerating1}
f_a^{(N)}(s,\frac{a\delta}{K})\leq d_a^{(a)}(s,\frac{\delta N}{K})\leq d_a^{(a)}(s,\frac{\delta}{\alpha}).
\end{equation}
For convenience, we write $h(s)=h_{a,N,K,\delta}(s):=f_a^{(N)}(s,\frac{a\delta}{K})$ and $\widehat{d}(s)=\widehat{d}_{a,\delta,\alpha}(s):=d_a^{(a)}(s,\frac{\delta}{\alpha})$, both of which are generating functions, satisfying $h(s)\leq \widehat{d}(s)$ for $N\geq a$.

 Let $J\geq0$ and $\kappa\in\{0,1,\cdots, a-1\}$ be such that $K=a J+\kappa$. Let $\mathcal{B}_{K}^{(N)}(\delta)$ be the collection of vertices $\sigma$ in $T^{(N)}$ such that
\begin{equation}
\frac{\kappa+aj}{K}\delta<x_{\sigma_{\kappa+aj+1}}<\cdots<x_{\sigma_{\kappa+aj+j}}\leq \frac{\kappa+aj+a}{K}\delta,\quad \forall j\in\{0,\cdots, J-1\},
\end{equation}
and that
\begin{equation}
0<x_{\sigma_1}<\cdots<x_{\sigma_{\kappa}}\leq\frac{\kappa}{K}\delta.
\end{equation}
where $K:=|\sigma|$.
According to the definition of $\mathcal{B}_{K}^{(N)}(\delta)$, one sees that
\begin{equation}
Z_{N,K}(\delta)\geq\#\mathcal{B}_{K}^{(N)}(\delta)=\sum_{|\omega|=\kappa}1_{(0<x_{\omega_1}<\cdots<x_{\omega}\leq\frac{\kappa}{K}\delta)}\sum_{|\sigma|=K}1_{(\sigma_\kappa=\omega)}1_{(\sigma\in\mathcal{B}_K^{(N)})},
\end{equation}
where given $\{0<x_{\omega_1}<\cdots<x_{\omega}\leq\frac{\kappa}{K}\delta\}$, the generating function of $\sum_{|\sigma|=K}1_{(\sigma_\kappa=\omega)}1_{(\sigma\in\mathcal{B}_K^{(N)})}$ is $\underbrace{h\circ\cdots\circ h}_{J}=:h^{\circ J}$. As a consequence,
\begin{eqnarray*}
&& \P_0\Big[Z_{N,K}(\delta)\geq1\Big]\geq\P_0\Big[\#\mathcal{B}_{K}^{(N)}(\delta)\geq1\Big]\\
&\geq&\P_0\Big[\sum_{|\omega|=\kappa}1_{(0<x_{\omega_1}<\cdots<x_{\omega}\leq\frac{\kappa}{K}\delta)}\geq1\Big]\P_0\Big[\sum_{|\sigma|=K}1_{(\sigma_\kappa=\omega)}1_{(\sigma\in\mathcal{B}_K^{(N)})}\geq1\Big\vert 0<x_{\omega_1}<\cdots<x_{\omega}\leq\frac{\kappa}{K}\delta\Big]\\
&=&\Big(1-f_\kappa^{(N)}(0,\frac{\kappa}{K}\delta)\Big)\Big(1-h^{\circ J}(0)\Big),
\end{eqnarray*}
since the generating function of $Z_{N,\kappa}(\frac{\kappa}{K}\delta)=\sum_{|\omega|=\kappa}1_{(0<x_{\omega_1}<\cdots<x_{\omega}\leq\frac{\kappa}{K}\delta)}$ is $f_\kappa^{(N)}(s,\frac{\kappa}{K}\delta)$. Applying the inequality (\ref{InPcouplinggenerating1}) to $f_\kappa^{(N)}(0,\frac{\kappa}{K}\delta)$ and $h$, respectively, shows that
\begin{eqnarray}\label{InPcoupling}
\P_0\Big[Z_{N,K}(\delta)\geq1\Big]&\geq& (1-d_\kappa^{(\kappa)}(0,\delta/\alpha))\Big(1-\big(\widehat{d}\ \big)^{\circ J}(0)\Big),
\end{eqnarray}
where $(\widehat{d}\ )^{\circ J}:=\underbrace{\widehat{d}\circ\cdots\circ \widehat{d}}_{J}$.
Going back to the generating function $\widehat{d}(s)=d_a^{(a)}(s,\frac{\delta}{\alpha})=\E_{\delta/\alpha}[s^{D_a^{(a)}}]$, we see that
\begin{equation}
\E_{\delta/\alpha}[D_a^{(a)}]=\frac{(a\delta/\alpha)^a}{a!}=(e_a\delta/\alpha)^a,
\end{equation}
where $e_a:=(\frac{a^a}{a!})^{1/a}$. By (\ref{InPeqstirling}), $e_a\uparrow e$ as $a\uparrow\infty$. For $\delta>\alpha/e$, there exists an integer $a(\delta,\alpha)$ such that $e_a\delta/\alpha>1$ for all $a\geq a(\delta,\alpha)$. This implies that
\begin{equation}
\widehat{d}^{\ \prime}(1)=\E_{\delta/\alpha}[D_a^{(a)}]>1,\quad\forall a\geq a(\delta,\alpha).
\end{equation}
Thus, for the Galton-Watson tree whose offspring has generating function $\widehat{d}(s)$, its extinction probability, denoted by $\widehat{q}(a,\delta/\alpha)$, satisfies that
\begin{equation}
\widehat{q}(a,\delta/\alpha)=\lim_{J\rightarrow\infty}\big(\widehat{d}\ \big)^{\circ J}(0)<1.
\end{equation}
This tells us that
\begin{equation}
\Big(1-\big(\widehat{d}\ \big)^{\circ J}(0)\Big)\geq 1-\widehat{q}(a,\delta/\alpha)=:\widehat{p}(a,\delta/\alpha)>0,\quad\forall J\geq0.
\end{equation}
Moreover, for any $a>0$ fixed, we have
\begin{equation}
\beta(a,\delta/\alpha):=\inf_{0\leq \kappa<a}\Big(1-d_\kappa^{(\kappa)}(0,\delta/\alpha)\Big)>0,
\end{equation}
as $d_\kappa^{(\kappa)}$ are non-trivial generating functions.

Therefore, we end up with
\begin{equation}\label{InPinfprob}
\inf_{N\geq a(\delta,\alpha)}\P\Big[Z_{N,\alpha N}(\delta)\geq1\Big]\geq c_0(\delta,\alpha)>0,
\end{equation}
where $c_0(\delta,\alpha):=\beta(a(\delta,\alpha),\delta/\alpha)\widehat{p}(a(\delta,\alpha),\delta/\alpha)>0$.

Notice that $\P[Z_{N,\alpha N}(\delta)>0]>0$ for any $1\leq N\leq a(\delta,\alpha)$. We conclude the proof of this lemma by taking $c(\delta,\alpha):=\min_{1\leq N\leq a(\delta,\alpha)}\{c_0(\delta,\alpha),\P[Z_{N,\alpha N}(\delta)>0]\}>0$.
\end{proof}

Now we are ready to prove the convergence (\ref{eq:criticalate}).

\begin{proof}[Proof of (\ref{eq:criticalate}).]
For $0<\alpha<e$, let $\delta\in(\frac{\alpha}{e},1\wedge \alpha)$. Observe that under $\P_0$,
\begin{equation}
Z_{N,\alpha N}\geq\sum_{|\omega|=1}1_{(0<x_\omega<1-\delta)}\sum_{|\sigma|=\alpha N}1_{(\sigma_1=\omega)}1_{(1-\delta<x_{\sigma_2}<\cdots<x_\sigma\leq1)}.
\end{equation}
For all vertex $\omega$ in the first generation, the variables $\sum_{|\sigma|=\alpha N}1_{(\sigma_1=\omega)}1_{(1-\delta<x_{\sigma_2}<\cdots<x_\sigma\leq1)}$ are independent and distributed as $Z_{N,\alpha N-1}(\delta)$. Consequently,
\begin{eqnarray*}
\P_0\Big[Z_{N,\alpha N}=0\Big]&\leq& \P_0\Big[\sum_{|\omega|=1}1_{(0<x_\omega<1-\delta)}\sum_{|\sigma|=\alpha N}1_{(\sigma_1=\omega)}1_{(1-\delta<x_{\sigma_2}<\cdots<x_\sigma\leq1)}=0\Big]\\
&= & \bigg(\P_0\Big[1_{(0<x_\omega<1-\delta)}\sum_{|\sigma|=\alpha N}1_{(\sigma_1=\omega)}1_{(1-\delta<x_{\sigma_2}<\cdots<x_\sigma\leq1)}=0\Big]\bigg)^N\\
&=&\bigg(\delta+(1-\delta)\P_0\Big[Z_{N,\alpha N-1}(\delta)=0\Big]\bigg)^N.
\end{eqnarray*}
By Lemma \ref{lem:InPuniformlowerbd}, $\P_0\Big[Z_{N,\alpha N-1}(\delta)=0\Big]\leq\P_0\Big[Z_{N,\alpha N}(\delta)=0\Big]\leq 1-c(\delta,\alpha)$. Thus,
\begin{equation}
\P_0\Big[Z_{N,\alpha N}=0\Big]\leq \bigg(\delta+(1-\delta)\Big(1-c(\delta,\alpha)\Big)\bigg)^N\leq e^{-c(\delta,\alpha)(1-\delta) N},
\end{equation}
which converges to zero as $N$ goes to infinity. This tells us that
\begin{equation}
\lim_{N\rightarrow\infty}\P_0\Big[Z_{N,\alpha N}\geq1\Big]=1,
\end{equation}
which is what we need.
\end{proof}

\section{The second order of $Z_{N,\alpha N}$ for $\alpha\in(0,e)$}
\label{ss:secondorder}

Note that for $\alpha\in(0,e)$, the population size $Z_{N,\alpha N}$ is asymptotically of order $e^{\theta(\alpha)N}$. The figure of the limit function $\theta(\alpha)$ is shown in {\sc Figure} \ref{fig4} at the end of this paper. 

We have the following lemma, concerning the second moment of $Z_{N,\alpha N}$.

\begin{lemma}\label{InPsecondmoment}
For $x\in[0,1)$ fixed and $0<\alpha<2(1-x)$, we have
\begin{equation}
\lim_{N\rightarrow\infty}\frac{\E_x\big[\big(Z_{N,\alpha N}\big)^2\big]}{m_{\alpha N}(x)^2}=\frac{2(1-x)}{2(1-x)-\alpha},
\end{equation}
where for any $k\geq1$ and $x\in[0,1]$,
\begin{equation}
m_k(x):=\frac{(1-x)^kN^k}{k!}=\E_x\big[Z_{N,k}\big].
\end{equation}
\end{lemma}

This lemma shows that under $\P_0$, for $\alpha\in(0,2)$, with positive probability, $Z_{N,\alpha N}$ is of the same order as its expectation $\E_0\Big(Z_{N,\alpha N}\Big)$, that is $N^{-1/2}e^{\theta(\alpha) N}$. But we do not get the second order of $Z_{N,\alpha N}$ for $\alpha\in[2,e)$. From the arguments as above, one can say that for $\alpha\in[2,e)$, with positive probability under $\P_0$,
\begin{equation}
c_{24}N^{-3/2}e^{\theta(\alpha) N}\leq Z_{N,\alpha N}\leq c_{25}N^{-1/2}e^{\theta(\alpha) N}.
\end{equation}

In particular, one sees that the maximum of $\alpha\mapsto\theta(\alpha)$ is reached at $\alpha=1$. We turn to consider $Z_{N,\alpha N}$ when $\alpha=1$. Let $\mathcal{L}(X, \P_x)$ denote the law of random variable $X$ under $\P_x$. The theorem is given as follows.

\begin{proposition}\label{InPmaxpopulation}
Let $\lambda>0$ fixed. Then the following convergence in law holds as $N\rightarrow\infty$:
\begin{equation}
\mathcal{L}\bigg(\frac{Z_{N,N}}{m_N}; \P_{\frac{\lambda}{N}}\bigg)\rightarrow e^{-\lambda}\times W,
\end{equation}
where $W$ is an exponential variable with mean $1$ and $m_N:=\frac{N^N}{N!}$.
\end{proposition}


\begin{remark}
A similar result to Proposition \ref{InPmaxpopulation} has been given in \cite{Berestycki-Brunet-Shi2013} by considering the accessible paths in the $N$-dimensional hypercubes. Our proof is mainly inspired by it.
\end{remark}

\subsection{The second moment of $Z_{N,\alpha N}$}

\begin{proof}[Proof of Lemma \ref{InPsecondmoment}.]
By (\ref{InPnewsecondmom}),
\begin{eqnarray}\label{InPeqsecondmom}
\E_x\big[\Big(Z_{N,k}\Big)^2\big]&=&m_k(x)+\frac{N-1}{N}\sum_{q=0}^{k-1} N^{2k-q}\P_x(\sigma,\sigma^\prime\in\mathcal{A}_{N,k}\Big\vert |\sigma\wedge\sigma^\prime|=q)\nonumber\\
&=&m_k(x)+\frac{N-1}{N}\sum_{q=0}^{k-1}\frac{((1-x)N)^{2k-q}}{(2k-q)!}{2(k-q)\choose (k-q)}\nonumber\\
&=&m_K(x)+m_K(x)^2\frac{N-1}{N}\sum_{q=0}^{k-1}a_k(q,x),
\end{eqnarray}
where $a_k(q,x):=\frac{(2k-2q)!k!k!}{[(1-x)N]^q (2k-q)!(k-q)!(k-q)!}$. Note that if $k+1\leq 2(1-x)N$,
\begin{equation}\label{InPmonoofa}
a_k(q+1,x)=a_k(q,x)\frac{(k-q)(2k-q)}{2(1-x)N(2k-2q-1)}\leq a_k(q,x),\quad\forall 0\leq q<k.
\end{equation}
Moreover, for $q\ll \sqrt{k}$ and $k=\alpha N$,
\begin{equation}\label{InPdominatedpart}
a_k(q,x)=\Big(\frac{k}{2(1-x)N}\Big)^{q}\frac{[(1-\frac{1}{k})\cdots(1-\frac{q-1}{k})]^2}{(1-\frac{q}{2k})\cdots(1-\frac{2q-1}{2k})}
=\Big(\frac{\alpha}{2(1-x)}\Big)^q[1+O(\frac{q^2}{k})].
\end{equation}\label{InPzeropart}
Take $q_0=\lceil\frac{2\log N}{\log (2(1-x))-\log\alpha}\rceil$ so that $\Big(\frac{\alpha}{2(1-x)}\Big)^{q_0}\leq N^{-2}$. It follows from (\ref{InPmonoofa}) that
\begin{equation}
\sum_{q=q_0}^{k-1}a_k(q,x)\leq k a_k(q_0,x)\leq c_{13}\alpha N^{-1},
\end{equation}
which vanished as $N$ goes to infinity. The dominated convergence theorem implies that for $0<\alpha<2(1-x)$ and $k=\alpha N$,
\begin{equation}
\lim_{N\rightarrow\infty}\sum_{q=0}^{q_0}a_k(q,x)=\sum_{q=0}^\infty \Big(\frac{\alpha}{2(1-x)}\Big)^q= \frac{2(1-x)}{2(1-x)-\alpha}.
\end{equation}
Moreover, $1/m_{\alpha N}(x)\rightarrow 0$ as $N$ goes to infinity. We thus conclude that for $0<\alpha<2(1-x)$,
\begin{equation*}
\lim_{N\rightarrow\infty}\frac{\E_x\big[\Big(Z_{N,\alpha N}\Big)^2\big]}{m_{\alpha N}(x)^2}=\frac{2(1-x)}{2(1-x)-\alpha}. \qedhere
\end{equation*}
\end{proof}

As a consequence of Lemma \ref{InPsecondmoment}, one sees that $\E_{\lambda/N}\Big[\Big(\frac{Z_{N,N}}{m_N}\Big)^2\Big]\rightarrow e^{-2\lambda}$ as $N\rightarrow\infty$.

\subsection{Proof of Proposition \ref{InPmaxpopulation}}

In this subsection, we investigate $Z_{N,N}$. Let $\{\mathcal{F}_k; k\geq1\}$ denote the natural filtration of the accessibility percolation on $N$-ary tree, i.e., $\mathcal{F}_k:=\sigma\{(\omega, x_\omega); \omega\in T^{(N)}, |\omega|\leq k\}$.

We introduce the following variables:
$$\theta_{N,k}(x):=\E_x\Big[Z_{N,N}\vert \mathcal{F}_k\Big],\quad\textrm{ and } \widetilde{\theta}_{N,k}(x):=\E_x\Big[Z_{N,N+k}\vert \mathcal{F}_k\Big].$$
Let $\theta:=Z_{N,N}$ for simplicity.

Recall that $m_N=\E_0[\theta]=\frac{N^N}{N!}$. We begin with the following lemma.
\begin{lemma}\label{InPlemmacoupleone}
As $N$ goes to infinity then $k$ goes to infinity,
\begin{equation}
\mathcal{L}\Big(\frac{\theta_{N,k}(\lambda/N)-\theta}{m_N},\P_{\lambda/N}\Big)\rightarrow 0.
\end{equation}
\end{lemma}

\begin{proof}[Proof.]
We observe that for any $z\in\r$ and $\delta>0$,
\begin{eqnarray*}
&&\P_x[\theta_{N,k}(x)\leq (z-\delta)m_N\vert\mathcal{F}_k]-\P_x[|\theta-\theta_{N,k}(x)|\geq m_N\delta\vert\mathcal{F}_k]\leq \P_x[\theta\leq m_Nz\vert\mathcal{F}_k];\\
&&\P_x[\theta_{N,k}(x)\leq (z+\delta)m_N\vert\mathcal{F}_k]+\P_x[|\theta-\theta_{N,k}(x)|\geq m_N\delta\vert\mathcal{F}_k]\geq \P_x[\theta\leq m_Nz\vert\mathcal{F}_k].
\end{eqnarray*}
Note also that
\begin{equation}
\P_x[|\theta-\theta_{N,k}(x)|\geq m_N\delta\vert\mathcal{F}_k]\leq \frac{\mathrm{Var}_x(\theta\vert\mathcal{F}_k)}{m_N^2\delta^2}.
\end{equation}
Consequently,
\begin{eqnarray*}
&&\P_x[\theta_{N,k}(x)\leq (z-\delta)m_N]- \P_x[\theta\leq m_Nz]\leq \E_x\Big[\frac{\mathrm{Var}_x(\theta\vert\mathcal{F}_k)}{m_N^2\delta^2}\Big];\\
&&\P_x[\theta\leq m_Nz]-\P_x[\theta_{N,k}(x)\leq (z+\delta)m_N]\leq \E_x\Big[\frac{\mathrm{Var}_x(\theta\vert\mathcal{F}_k)}{m_N^2\delta^2}\Big].
\end{eqnarray*}
Thus, it suffices to prove the following convergence.
\begin{equation}\label{InPconvergeconditionalvar}
\lim_{k\rightarrow\infty}\lim_{N\rightarrow\infty}\frac{\E_{\lambda/N}[\mathrm{Var}(\theta\vert\mathcal{F}_k)]}{m_N^2}= 0.
\end{equation}

The branching property yields that
\begin{equation}
\mathrm{Var}_x(\theta\vert\mathcal{F}_k)=\sum_{\sigma\in \mathcal{A}_{N,k}} v(x_\sigma, N-k),
\end{equation}
where $v(y,L):=\E_y[(Z_{N,L})^2]-\E_y[Z_{N,L}]^2$ for any $L\geq1$. Taking the expectation implies that
\begin{eqnarray}\label{InPconditionalvar}
\E_x[\mathrm{Var}_x(\theta\vert\mathcal{F}_k)]&=& N^k \int_x^1 dy \frac{y^{k-1}}{(k-1)!}v(y, N-k).
\end{eqnarray}
By (\ref{InPeqsecondmom}), we have
\begin{eqnarray*}
v(y,L)&=&m_L(y)+m_L(y)^2\frac{N-1}{N}\sum_{q=0}^{L-1}a_L(q,y)-m_L(y)^2\\
         &=& m_L(y)+m_L(y)^2\frac{N-1}{N}\sum_{q=1}^{L-1}a_L(q,y)-\frac{m_L(y)^2}{N}.
\end{eqnarray*}
Plugging it into (\ref{InPconditionalvar}) yields that
\begin{eqnarray*}
\E_x[\mathrm{Var}_x(\theta\vert\mathcal{F}_k)]&=&m_N(x)+m_N(x)^2\frac{N-1}{N}\sum_{q=k+1}^{N-1}a_N(q,x)-\frac{1}{N}m_N(x)^2 a_N(k,x).
\end{eqnarray*}
It follows from (\ref{InPdominatedpart}) and (\ref{InPzeropart}) that $\sum_{q=k+1}^{N-1}a_N(q,\lambda/N)\rightarrow\frac{1}{2^k}$. Clearly, $m_N(\lambda/N)/m_N=(1-\lambda/N)^N\rightarrow e^{-\lambda}$. Therefore,
\begin{equation}
\lim_{N\rightarrow\infty}\frac{\E_{\lambda/N}[\mathrm{Var}_{\lambda/N}(\theta\vert\mathcal{F}_k)]}{m_N^2}=\frac{e^{-2\lambda}}{2^k},
\end{equation}
which vanishes as $k$ goes to infinity. This yields (\ref{InPconvergeconditionalvar}) and completes the proof of Lemma \ref{InPlemmacoupleone}.
\end{proof}

\begin{lemma}\label{InPlemmacoupletwo}
For any $k\geq0$ fixed, we have
\begin{equation}
\lim_{N\rightarrow\infty}\frac{\E_{\lambda/N}[(\theta_{N,k}(\lambda/N)-\widetilde{\theta}_{N,k}(\lambda/N))^2]}{m_N^2}=0.
\end{equation}
\end{lemma}

\begin{proof}[Proof.]
By Jensen's inequality,
\begin{equation}
\Big(\theta_{N,k}(x)-\widetilde{\theta}_{N,k}(x)\Big)^2=\Big(\E_x[Z_{N,N+k}-Z_{N,N}\vert\mathcal{F}_k]\Big)^2\leq \E_x\Big[\Big(Z_{N,N+k}-Z_{N,N}\Big)^2\big\vert\mathcal{F}_k\Big].
\end{equation}
Taking the expectation yields that
\begin{equation}
\E_x\Big[\Big(\theta_{N,k}(x)-\widetilde{\theta}_{N,k}(x)\Big)^2\Big]\leq \E_x\Big[\Big(Z_{N,N+k}-Z_{N,N}\Big)^2\Big],
\end{equation}
which, by the Cauchy-Schwarz inequality, is bounded by
\begin{equation}
k\sum_{i=1}^k \E_x\Big[\Big(Z_{N,N+i}-Z_{N,N+i-1}\Big)^2\Big].
\end{equation}
Let $L=K+i-1\geq K$. Then,
\begin{equation}
Z_{N,L+1}-Z_{N,L}=\sum_{\sigma\in \mathcal{A}_{N,L}}(y_{\sigma}-1),
\end{equation}
where $y_\sigma:=\sum_{|\omega|=L+1}1_{(\omega_L=\sigma)}1_{(x_\omega>x_{\sigma})}$. It immediately follows that
\begin{eqnarray}\label{InPconsecutivedif}
&&\Big(Z_{N,L+1}-Z_{N,L}\Big)^2=\sum_{\sigma\in \mathcal{A}_{N,L}}(y_\sigma-1)^2+\sum_{\sigma\neq\sigma^\prime; \sigma,\sigma^\prime\in \mathcal{A}_{N,L}}(y_\sigma-1)(y_{\sigma^\prime}-1)\nonumber\\
&=&\sum_{\sigma\in \mathcal{A}_{N,L}}(y_\sigma-1)^2+\sum_{q=0}^{L-1}\sum_{|\sigma\wedge\sigma^\prime|=q}1_{(\sigma,\sigma^\prime\in \mathcal{A}_{N,L})}(y_\sigma-1)(y_{\sigma^\prime}-1),
\end{eqnarray}
where $\sigma\wedge\sigma^\prime$ is, as before, the latest common ancestor of $\sigma$ and $\sigma^\prime$. Note that under $\P_x[\cdot\vert\mathcal{F}_L]$, $y_\sigma$'s are independent binomial variables with parameters $N$ and $1-x_\sigma$. Thus, taking $\E_x[\cdot\vert\mathcal{F}_L]$ on both sides of (\ref{InPconsecutivedif}) yields that
\begin{equation}\label{InPsigmalplussigma2}
\E_x\Big[\Big(Z_{N,L+1}-Z_{N,L}\Big)^2\Big\vert\mathcal{F}_L\Big]=\Sigma_1+\Sigma_2,
\end{equation}
where
\begin{eqnarray}
\Sigma_1&:=&\sum_{\sigma\in \mathcal{A}_{N,L}}\E_x\Big[(y_\sigma-1)^2\Big\vert\mathcal{F}_L\Big];\\
\Sigma_2&:=&\sum_{q=0}^{L-1}\sum_{|\sigma\wedge\sigma^\prime|=q}1_{(\sigma,\sigma^\prime\in \mathcal{A}_{N,L})}(N(1-x_\sigma)-1)(N(1-x_{\sigma^\prime})-1).
\end{eqnarray}
Obviously, $(y_\sigma-1)^2\leq N^2$. Hence,
\begin{equation}\label{InPsigma1upp}
\E_x[\Sigma_1]\leq N^2\E_x[Z_{N,L}]=N^2 m_L(x)=o(m_N^2).
\end{equation}
Conditioning on the value of $x_{\sigma\wedge\sigma^\prime}$ yields that
\begin{eqnarray}\label{InPsigma2upp}
\E_x[\Sigma_2] &=&\E_x \Big[\sum_{q=0}^{L-1}\sum_{|\sigma\wedge\sigma^\prime|=q}1_{(\sigma,\sigma^\prime\in \mathcal{A}_{N,L})}(N(1-x_\sigma)-1)(N(1-x_{\sigma^\prime})-1)\Big]\nonumber\\
         &=&\frac{N-1}{N}\sum_{q=0}^{L-1}N^{2L-q}\int_{x}^1dy\frac{(y-x)^{q-1}}{(q-1)!}\bigg[\int_y^1dx_{\sigma}\frac{(x_\sigma-y)^{L-q-1}(N(1-x_\sigma)-1)}{(L-q-1)!}\bigg]^2\nonumber\\
         &=&\frac{N-1}{N} \sum_{q=0}^{L-1}(\delta_1(q)-2\delta_2(q)+\delta_3(q)),
\end{eqnarray}
where
\begin{eqnarray*}
\delta_1(q):&=&N^{2L-q}\int_{x}^1dy\frac{(y-x)^{q-1}}{(q-1)!}\bigg(\frac{N(1-y)^{L-q+1}}{(L-q+1)!}\bigg)^2;\\
\delta_2(q):&=&N^{2L-q}\int_{x}^1dy\frac{(y-x)^{q-1}}{(q-1)!}\bigg(\frac{N(1-y)^{L-q+1}}{(L-q+1)!}\times\frac{(1-y)^{L-q}}{(L-q)!}\bigg);\\
\delta_3(q):&=&N^{2L-q}\int_{x}^1dy\frac{(y-x)^{q-1}}{(q-1)!}\bigg(\frac{(1-y)^{L-q}}{(L-q)!}\bigg)^2.
\end{eqnarray*}
 On the one hand,
\begin{equation}
0\leq\delta_1(q)-2\delta_2(q)+\delta_3(q)\leq5\delta_3(q),\quad \forall q\geq0.
\end{equation}
On the other hand,
\begin{equation}
\delta_1(q)-2\delta_2(q)+\delta_3(q)\leq \delta_3(q)O({\frac{q^2}{L^2}}),\quad \forall q\leq O({\log L}).
\end{equation}
Thus, (\ref{InPsigma2upp}) becomes
\begin{equation}
\E_x[\Sigma_2]\leq\sum_{q=c_{14}\log L}^{L-1}5\delta_3(q)+\sum_{q=0}^{c_{14}\log L}\delta_3(q)c_{15}({\frac{q^2}{L^2}}).
\end{equation}
Notice that $\delta_3(q)=m_L^2(x)a_L(q,x)$. Take $x=\lambda/N$ and recall that $L=N+i-1$. By (\ref{InPmonoofa}), for $N$ large enough so that $N+i\leq 2(1-\lambda/N)N$, $a_L(q,x)$ is non-increasing as $q$ increases. It follows that
\begin{eqnarray*}
\E_x[\Sigma_2]&\leq&m^2_L(x)\Big(\sum_{q=c_{14}\log L}^{L-1}5a_L(q,x)+\sum_{q=0}^{c_{14}\log L}a_L(q,x)c_{15}({\frac{q^2}{L^2}})\Big)\\
&\leq&m^2_L(x)\Big(5La_L(c_{14}\log L, x)+c_{15}\frac{(c_{14}\log L)^3}{L^2}a_L(0,x)\Big).
\end{eqnarray*}
Note that $a_L(0,x)=1$. By (\ref{InPdominatedpart}), $a_L(c_{14}\log L, x)=\Big(\frac{L}{2(1-x)N}\Big)^{c_{14}\log L}[1+O(\frac{(\log L)^2}{L})]$. We can choose a suitable $c_{14}$ so that $a_L(c_{14}\log L,x)=o(N^{-1})$. As a result,
\begin{equation}\label{InPsigma2uppbd}
\E_x[\Sigma_2]=m^2_L(x) o_N(1)=m_N^2 o_N(1).
\end{equation}
We return to (\ref{InPsigmalplussigma2}). Combining (\ref{InPsigma1upp}) with (\ref{InPsigma2uppbd}) implies that
\begin{equation}
\E_{\lambda/N}\Big[\frac{\Big(Z_{L+1}^{(N)}-Z_{L}^{(N)}\Big)^2}{m_N^2}\Big]=o_N(1).
\end{equation}
Therefore, for any $k\geq1$ fixed, we have
\begin{equation*}
\lim_{N\rightarrow\infty}\E_{\lambda/N}[\frac{(\theta_{N,k}(\lambda/N)-\widetilde{\theta}_{N,k}(\lambda/N))^2}{m_N^2}]= 0.\qedhere
\end{equation*}
\end{proof}

By considering the variables $\widetilde{\theta}_{N,k}(x)$, we will prove the convergence in law in Proposition \ref{InPmaxpopulation} as follows.

\begin{proof}[Proof of Proposition \ref{InPmaxpopulation}.]
In view of Lemmas \ref{InPlemmacoupleone} and \ref{InPlemmacoupletwo}, we only need to prove that the distribution $\mathcal{L}\Big(\frac{\widetilde{\theta}_{N,k}(\lambda/N)}{m_N},\ \P_{\lambda/N}\Big)$ converges weakly to an exponential variable of mean $e^{-\lambda}$, as $N$ goes to infinity then $k$ goes to infinity.

 Clearly, $\widetilde{\theta}_{N,0}(x)=m_N(1-x)^N$ with $m_N=\frac{N^N}{N!}$. Define for any $k\geq0$ and $\mu\geq0$,
\begin{equation}\label{InPdefofg}
G_k(\mu,x,N):=\E_x\Big[\exp\{-\mu \widetilde{\theta}_{N,k}(x)/m_N\}\Big],
\end{equation}
which is the Laplace transform of $\frac{\widetilde{\theta}_{N,k}(x)}{m_N}$.

It is immediate that $G_0(\mu, x, N)=\exp\{-\mu (1-x)^N\}$. Recursively,
\begin{equation}
\widetilde{\theta}_{N,k+1}(x)=\E_x\big[Z_{N,N+k+1}\big\vert\mathcal{F}_{k+1}\big]=\sum_{\sigma\in \mathcal{A}_{N,1}}\E_{x_\sigma}\Big[Z_{N,N+k}\Big\vert\mathcal{F}_k\Big]=\sum_{|\sigma|=1}1_{(x_\sigma>x)}\widetilde{\theta}_{N,k}(x_\sigma),
\end{equation}
where for $|\sigma|=1$, $1_{(x_\sigma>x)}\widetilde{\theta}_{N,k}(x_\sigma)$ are i.i.d. It follows that
\begin{equation}\label{InPrecursiveeq}
 G_{k+1}(\mu, x, N)=\Big[x+\int_x^1dy G_k(\mu, y, N)\Big]^N.
\end{equation}

We define for $\lambda$, $\mu>0$,
\begin{eqnarray}
Q_0(\mu, \lambda):&=&\exp\{-\mu e^{-\lambda}\};\\
Q_{k+1}(\mu, \lambda):&=& \exp\Big\{-\int_\lambda^\infty\big(1-Q_k(\mu, y)\big)dy\Big\},\quad \forall k\geq0.
\end{eqnarray}
Clearly, $\lim_{N\rightarrow\infty}G_0(\mu, \frac{\lambda}{N}, N)=Q_0(\mu,\lambda)$. We are going to prove that for any $k\geq0$,
\begin{equation}\label{InPgeneratingfunclimit}
\lim_{N\rightarrow\infty}G_k(\mu,\frac{\lambda}{N}, N)=Q_k(\mu,\lambda).
\end{equation}

Suppose that (\ref{InPgeneratingfunclimit}) holds for $k\geq0$. By a change of variables, (\ref{InPrecursiveeq}) becomes that
\begin{equation}
G_{k+1}(\mu, \frac{\lambda}{N}, N)=\Big[1-\int_{\lambda}^N dy \Big(1-G_k(\mu, \frac{y}{N}, N)\Big)\Big]^N.
\end{equation}
Because $1-e^{-z}\leq z$ for all $z\in\r$, (\ref{InPdefofg}) gives that
\begin{equation}
0\leq 1-G_{k}(\mu, \frac{y}{N}, N)\leq \mu\E_{\frac{y}{N}}[Z_{N,N+k}]/m_N=\frac{\mu}{m_N}\frac{[N(1-y/N)]^{N+k}}{(N+k)!}\leq \mu e^{-y}.
\end{equation}
The dominated convergence theorem implies that
\begin{equation}
\int_{\lambda}^N dy \Big(1-G_k(\mu, \frac{y}{N}, N)\Big)\xrightarrow{N\rightarrow\infty}\int_\lambda^\infty dy\Big(1-Q_k(\mu,y)\Big).
\end{equation}
It follows that $\lim_{N\rightarrow\infty}G_{k+1}(\mu,\frac{\lambda}{N}, N)=Q_{k+1}(\mu,X)$. By induction, we conclude (\ref{InPgeneratingfunclimit}) for any $k\geq0$.

We write $Q_k(\mu, \lambda)=F_k(\mu e^{-\lambda})$ for all $k\geq0$. We check that
\begin{equation}\label{InPrecursiveeqofF}
F_{k+1}(z)=\exp\Big\{-\int_0^{z}\frac{1-F_k(u)}{u}du\Big\},\qquad F_0(z)=e^{-z}.
\end{equation}

Define $\Delta_k(z)$ for $z>-1$ and $z\neq 0$ by
\begin{equation}\label{InPdefofdelta}
\Delta_k(z):=2^k\frac{(1+z)^3}{z^2}\Big[\frac{1}{1+z}-F_k(z)\Big].
\end{equation}
Then we claim that there exists a constant $M$ such that for all $k\geq0$,
\begin{equation}\label{InPuniformbd}
0\leq \Delta_k(z)\leq M,\quad \forall z>-1.
\end{equation}
Indeed, for $k = 0$,
\begin{equation}
\Delta_0(z)=\frac{(1+z)^3}{z^2}\Big[\frac{1}{1+z}-e^{-z}\Big],
\end{equation}
which is nonnegative for $z>-1$, because $e^z\geq 1+z$. Moreover, since $\lim_{z\rightarrow 0}\Delta_0(z)=1/(2e)$, define $\Delta_0(0):=\frac{1}{2e}$ so that $\Delta_0(z)$ is continuous in $(-1,\infty)$, and that both $\lim_{z\downarrow-1}\Delta_0(z)$ and $\lim_{z\uparrow\infty}\Delta_0(z)$ exist and are bounded. Hence, there exists $M\in(0,\infty)$ such that
\begin{equation}
0\leq \Delta_0(z)\leq M,\quad \forall z>-1.
\end{equation}
Assume now that (\ref{InPuniformbd}) holds at order $k$. In view of (\ref{InPrecursiveeqofF}) and (\ref{InPdefofdelta}),
\begin{equation}
F_{k+1}(z)=\frac{1}{1+z}\exp\Big\{-\int_0^z\frac{u}{(1+u)^3}\frac{\Delta_k(u)}{2^k}du\Big\}.
\end{equation}
This leads to
\begin{equation*}
\frac{1}{1+z}\geq F_{k+1}(z)\geq \frac{1}{1+z}\Big[1-\int_0^z\frac{u}{(1+u)^3}\frac{M}{2^k}du\Big]= \frac{1}{1+z}\Big[1-\frac{M}{2^k}\frac{z^2}{2(1+z)^2}\Big].
\end{equation*}
This implies that (\ref{InPuniformbd}) holds for $k+1$. In view of (\ref{InPdefofdelta}) and (\ref{InPuniformbd}), we check that
\begin{equation}
\lim_{k\rightarrow\infty}F_k(z)=\frac{1}{1+z},\textrm{ for }z>-1.
\end{equation}
Recall that $Q_k(\mu, \lambda)=F_k(\mu e^{-\lambda})$. Going back to (\ref{InPgeneratingfunclimit}), we let $k$ go to infinity for both sides and obtain that for any $\lambda>0$ fixed,
\begin{equation}
\lim_{k\rightarrow\infty}\lim_{N\rightarrow\infty}\E_{\lambda/N}\Big[e^{-\mu\frac{\widetilde{\theta}_{N,k}(\lambda/N)}{m_N}}\Big]=\lim_{k\rightarrow\infty}\lim_{N\rightarrow\infty} G_k(\mu,\frac{\lambda}{N},N)=\frac{1}{1+\mu e^{-\lambda}},
\end{equation}
which is the Laplace transform of an exponential variable of mean $e^{-\lambda}$. Therefore, we deduce that as $N\rightarrow\infty$,
\begin{equation}
\mathcal{L}\Big(\frac{Z_{N,N}}{m_N},\P_{\lambda/N}\Big)\rightarrow e^{-\lambda}\times W,
\end{equation}
where $W$ is an exponential variable with mean $1$.
\end{proof}

 An analogous argument implies that for $0<\alpha<1$, started from $x=1-\alpha+\frac{\lambda}{N}$, $\mathcal{L}\Big(\frac{Z_{N,\alpha N}}{m_{\alpha N}(1-\alpha)},\P_{x}\Big)$ converges to an exponential distribution of mean $e^{-\lambda}$.

\begin{figure}[htp]
\begin{center}
\includegraphics[width=\textwidth]{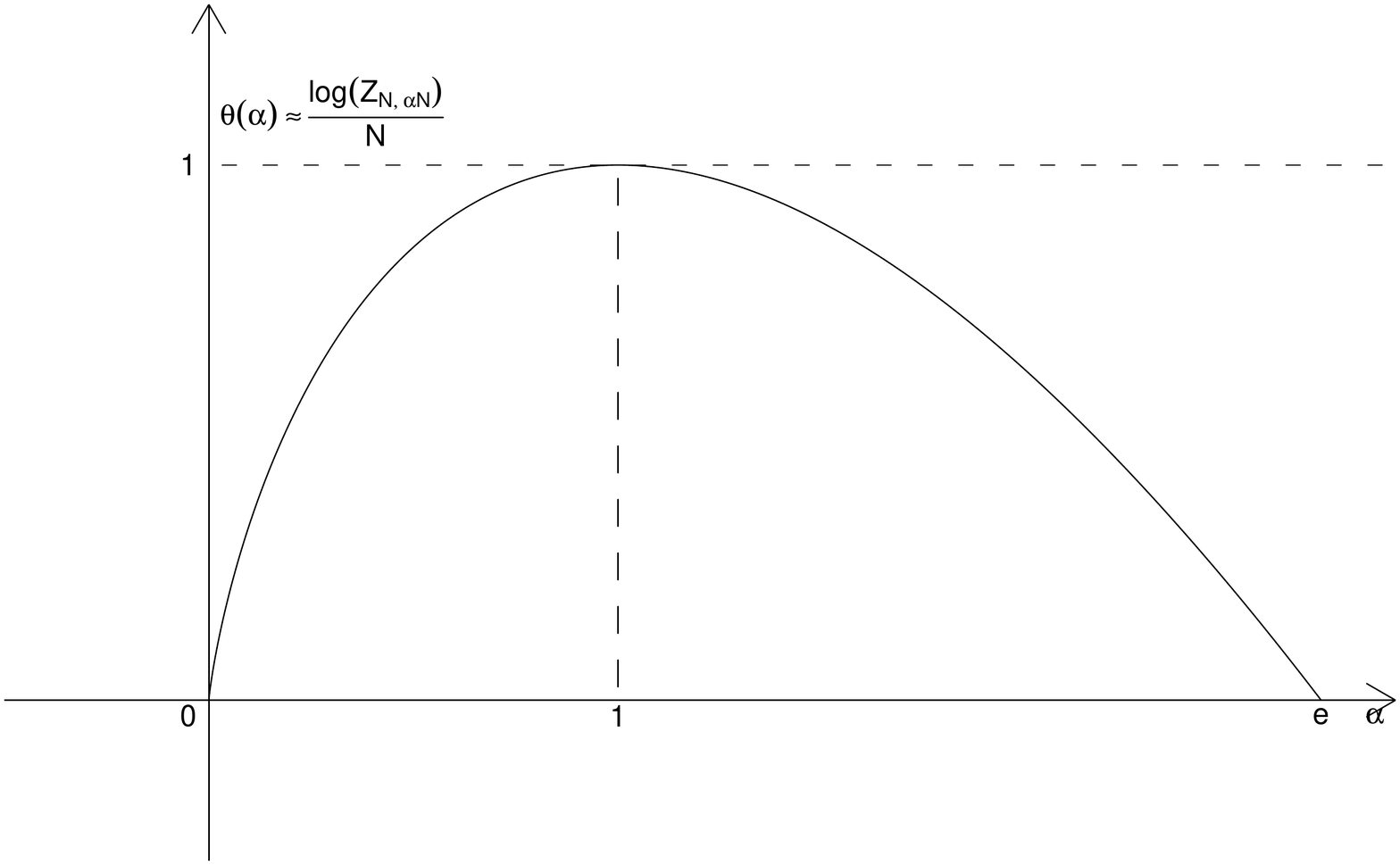}
\caption{The curve of $\alpha\mapsto\theta(\alpha)=\alpha(1-\log\alpha)$.}
\label{fig4}
\end{center}
\end{figure}

\bibliographystyle{plain}
\bibliography{bibli}

\begin{thebibliography}{1}

\bibitem{Aita-U-I-N-K-H2000}
T.~Aita, H.~Uchiyama, T.~Inaoka, M.~Nakajima, T.~Kokubo, and Y.~Husimi.
\newblock Analysis of a local fitness landscape with a model of the rough {M}t.
  {F}uji-type landscape: application to prolyl endopeptidase and thermolysin.
\newblock {\em Biopolymers}, 54(1):64--79, 2000.

\bibitem{Berestycki-Brunet-Shi2013}
J.~Berestycki, \'E Brunet, and Z.~Shi.
\newblock How many evolutionary histories only increase fitness?
\newblock arXiv: 1304.0246, 2013.

\bibitem{Franke-K-V-K2011}
J.~Franke, A.~K\"{o}zer, A.~J. G.~M. de~Visser, and J.~Krug.
\newblock Evolutionary accessibility of mutational pathways.
\newblock {\em PLos Comput. Biol.}, 7:e1002134, 2011.

\bibitem{Hegarty-Martinsson2013}
P.~Hegarty and A.~Martinsson.
\newblock On the existence of accessible paths in various models of fitness
  landscapes.
\newblock arXiv:1210.4798, 2013.

\bibitem{Kingman1978}
J.~F.~C. Kingman.
\newblock A simple model for the balance between selection and mutation.
\newblock {\em J. Appl. Probab.}, 15(1):1--12, 1978.

\bibitem{Nowak-Krug2013}
S.~Nowak and J.~Krug.
\newblock Accessibility percolation on $n$-trees.
\newblock {\em Europhys. Lett.}, 101(6):66004, 2013.

\bibitem{Roberts-Zhao2013}
M.~I. Roberts and L.~Z. Zhao.
\newblock Increasing paths in trees.
\newblock arXiv: 1305.0814, 2013.

\end{thebibliography}

\end{document}